\documentclass[10pt]{amsart}
\usepackage[centertags]{amsmath}
\usepackage{amsfonts}
\usepackage{amsthm}
\usepackage{newlfont}
\usepackage{amscd}
\usepackage{amsgen}
\usepackage{fancyvrb}
\usepackage[all]{xy}

\usepackage{amssymb}
\usepackage{longtable}
\usepackage{listings}
\usepackage{fullpage}
\newlength{\defbaselineskip} 

\newtheorem{thm}{Theorem}[section]
\newtheorem{cor}[thm]{Corollary}
\newtheorem{lemm}[thm]{Lemma}
\newtheorem{prop}[thm]{Proposition}
\newtheorem{conj}[thm]{Conjecture}
\newtheorem{prob}[thm]{Problem}
\newtheorem{claim}[thm]{Claim}

\theoremstyle{definition}
\newtheorem{defi}[thm]{Definition}
\newtheorem{rem}[thm]{Remark}

\numberwithin{equation}{section} 

\DeclareMathOperator{\Pic}{Pic}

\DeclareMathOperator{\Hom}{Hom}
\DeclareMathOperator{\Ext}{Ext}
\DeclareMathOperator{\Sym}{Sym}
\DeclareMathOperator{\sing}{sing}
\DeclareMathOperator{\supp}{supp}

\DeclareMathOperator{\sym}{Sym}

          \newcommand\PP{{\mathbb{P}}}
           
          \newcommand\C{{\mathcal C}}
          
           \newcommand\F{{\mathcal F}}
           \newcommand\G{{\mathcal G}}

          \newcommand\oo{\mathcal O}
          
         \newcommand\HOM{\mathcal{H}om}
          \newcommand\ext{\mathcal{E}xt}

\newcounter{appendice}

\begin{document}

\title{On IHS fourfolds with $b_2=23$}
\author{Grzegorz Kapustka}
\dedicatory{ (with an appendix written jointly with \textsc{Micha\l \ Kapustka})}
\thanks{
 The project was supported by MNSiW, N N201 414539. The Appendix is supported by
the program Iuventus plus ``Rangi i rangi brzegowe wielomian\'ow
oraz r\'ownania rozmaito\'sci siecznych".}
 
\begin{abstract}  The present work is concerned with the study of
four-dimensional irreducible holomorphic symplectic manifolds with  
second
Betti number $23$. We describe their birational geometry and their relations to EPW sextics.
\end{abstract}

\maketitle

\section{introduction}
By an \emph{irreducible holomorphic symplectic (IHS) fourfold} we mean (see \cite{B}) a
four-dimen\-sional simply connected K\"{a}hler manifold with trivial canonical
bundle that admits a unique (up to 
a constant)
closed 
non-degenerate holomorphic $2$-form and is not a product of two
manifolds. These manifolds are among the building blocks of K\"{a}hler fourfolds with
trivial first Chern class \cite[Thm.~2]{B}.
In the case of four-dimensional examples their second Betti number 
$b_2$  is
bounded and $3\leq b_2\leq 8$ or $b_2=23$ (see \cite{G}).
There are however only two known families of IHS's in this dimension, one with $b_2=7$
and the other with  $b_2=23$ \cite{B}.
The first is the deformation of the Hilbert scheme of two points on a K3 surface
and the second is the deformation of the Hilbert scheme of three points
that sum to $0$ on an abelian surface.

In this paper we address the problem of classification of IHS fourfolds $X$ with
$b_2=23$. This program was initiated by O'Grady whose purpose is to prove that
IHS fourfolds that are numerically equivalent to the Hilbert scheme of two points on a $K3$ 
surface 
are deformation equivalent to this Hilbert scheme (are of Type $K3^{[2]}$).

It is known 
from \cite{V} and \cite{G} that for IHS fourfolds with $b_2=23$ the cup product induces
an isomorphism
\begin{equation} \label{equa1}\sym^2 H^2(X,\mathbb{Q} )\simeq
H^4(X,\mathbb{Q})
\end{equation} 
and that
$H^3(X,\mathbb{Q})=0$.
By \cite{F} the Hodge diamond admits additional symmetries, and by \cite{S} it has
the following shape:
 $$1$$
    $$0 \quad\quad 0$$
    $$1 \qquad 21\qquad 1$$
    $$0 \qquad 0\qquad 0\qquad0$$
    $$1 \qquad 21\qquad 232\qquad 21\qquad 1$$
    $$0 \qquad 0\qquad 0\qquad0$$
    $$1 \qquad 21\qquad 1$$
     $$1$$

Recall that for an IHS fourfold $X$ we can find a (Fujiki) constant 
$c$ such that
for $\alpha\in H^2(X,\mathbb{Z})$, we have $cq(\alpha)^2=\int \alpha^{4}$ where
$q$ is a primitive integral quadric form
called the Beauville--Bogomolov form defining a lattice structure on
$H^2(X,\mathbb{Z})$ called the Beauville--Bogomolov (for short B-B) lattice.

 In order to classify IHS fourfolds with $b_2=23$ we have to find the possible
lattices and the possible Fujiki invariants for the given lattice. Next for
a fixed Fujiki invariant and B-B lattice find
all deformation families of IHS manifolds with the given numerical data.
Note that the lattices for the known examples are even but not unimodular.

The plan of the paper is the following. We show that each ample divisor on an IHS
fourfold $X$ with $b_2(X)=23$ has self-intersection 
which is an integer 
 of
the form $12k^2$ for some $k\in\mathbb{N}$.
Next we study the case when $X$ admits a divisor $H$ with $H^4=12$, i.e.~the
minimal possible self-intersection.
In this case 
 $h^0(\mathcal{O}_X(H))=6$  the first possibility to consider is when
$H$ defines a birational morphism $\varphi_{|H|}\colon X\to \PP^5$ into a hypersurface of degree $12$.
Recall that the ideal of the conductor of $\varphi_{|H|}$ then  defines a
scheme structure $C$ on the singular locus of the image $\varphi_{|H|}(X)\subset \PP^5$. It is known
that $C\subset \mathbb{P}^5$ is Cohen--Macaulay of pure dimension~$3$.

Recall that an \textit{EPW sextic} $S_A\subset\PP^5=:\PP(W)$ is a special sextic hypersurface defined as the determinant
of the morphism
\begin{equation}\label{eqEPW} A \otimes \mathcal{O}_{\mathbb{P}^5} \rightarrow
\Omega^2_{\mathbb{P}^5}(3)\subset \PP(W) \times \textstyle\bigwedge^3 W
\end{equation}
corresponding to the choice of a 10-dimensional Lagrangian $A\subset \bigwedge^3 W$ with respect to the natural symmetric form  (as in
\cite[Ex.~9.3]{EPW}).
Furthermore, following O'Grady we denote 
\begin{equation}\label{2.2}
\Theta_A=\{V\in G(3,W)\mid V\in G(3,W)\cap
\mathbb{P}(A)\subset \mathbb{P}(\textstyle\bigwedge^3 W) \} .
\end{equation}
The set $\Theta_A$ is empty for a generic choice of $A$ and generally measures how
singular the EPW sextic is. Recall that EPW sextics 
were also constructed by O'Grady \cite{O1} as  quotients by an involution of an IHS fourfold
deformation equivalent to $Hilb^2(S)$ where $S$ is a $K3$ surface that admits a polarization of degree $12$.
Our main result is the following:

\begin{thm}\label{EPW} 
Suppose that an IHS fourfold $X$ with $b_2=23$ admits an
ample divisor with $H^4=12$ such that $H$ defines a birational morphism
$\varphi_{|H|}$. Then there is a unique sextic containing the singular scheme $C\subset \PP^5$ of $\varphi_{|H|}(X)\subset \PP^5$ defined above. Moreover, this sextic
is an EPW sextic that we denote by $S_A$ (we call
it the EPW hypersurface  adjoint to the image $\varphi_{|H|}(X)\subset \PP^5$).
\end{thm}
When $H$ is fixed we denote  $\varphi:=\varphi_{|H|}$ and  $X'=\varphi(X)\subset \PP^5$.
Our approach to the study of the embedding $C\subset \PP^5$ is to use the methods of homological algebra
described in \cite{EFS}, \cite{EPW}.
In Section \ref{section-gen} we show that the unique adjoint EPW sextic $S_A$ obtained in Theorem \ref{EPW} has to be special.

\begin{prop}\label{general case} 
Suppose that an IHS fourfold $X$ with $b_2=23$ admits an
ample divisor with $H^4=12$ such that $H$ defines a birational morphism
$\varphi_{|H|}$. Then the sextic $S_{A}$ adjoint to the image $X'\subset \PP^5$ is an EPW sextic that is not generic.
More precisely, if  we denote by $\Theta_{A}$ the set defined by (\ref{2.2}), then 
 $\Theta_{A}\neq \emptyset$.
 \end{prop}
 The proposition above suggests in fact that the morphism $\varphi_{|H|}$ is
never birational.
Indeed,  Proposition \ref{general case} implies that for a fixed sextic $S_A$
which is adjoint to the image of an IHS manifold, there is an at least one-dimensional family of polarized IHS fourfolds $X$ such that
$S_A$ is the adjoint hypersurface to $\varphi_{|H|} (X)\subset \PP^5$.

 The idea of the proof of the proposition is the following:
Suppose that $S_A$ with $\Theta_A=\emptyset$ is the adjoint hypersurface to $X'\subset \PP^5$. Then we show that $S_A$ is normal and we construct a natural desingularization
$\pi\colon V\to S_A$ described in Section \ref{12}.
We obtain a contradiction by considering the pull-back $\pi^{\ast}(X'\cap S_A)$
on $V$ using the knowledge of the Picard group of $V$ and the natural duality of
%
$V$. In the Appendix we present technical results used in the proofs concerning the geometry of
the orbits of the natural
$PGL(6)$ action on $\PP(\bigwedge^3 \mathbb{C}^6)$.

This work is motivated by the study of the following question of 
 Beauville  \cite[q.~ 4]{B0}:

\begin{prob}
Is each IHS fourfold with $b_2=23$ deformation equivalent 
 to $Hilb^2(S)$ where $S$ is a $K3$ surface?
\end{prob}

More precisely, we are motivated by the special case of the above question, called the O'Grady conjecture \cite{O}: Show that if an IHS fourfold $X$ is numerically
equivalent to $S^{[2]}$ where $S$ is a $K3$ surface (i.e.~the Fujiki invariant $c$  is 3 and $(H^2(X,\mathbb{Z}),q)$ is
isometric to $U^3\oplus E_8^2\oplus \langle -2\rangle$ with the standard
notation) then it is deformation equivalent to it.

If an IHS fourfold $X$ 
satisfies the assumptions of the above O'Grady conjecture then we have $b_2(X)=23$ and it is proven in \cite{O} that
$X$ is either of type $K3^{[2]}$ or is deformation equivalent to a polarized manifold
$(X_0,H_0)$ (satisfying the conditions of \cite[Claim 4.4]{O6}) such that 
$\varphi_{|H_0|}$ is a birational map 
whose image is a hypersurface of degree $6\leq d\leq 12$. 
 O'Grady conjectured
that the latter case cannot happen. In \cite{K} we showed that $d\geq 9$ and
that $|H_0|$ has at most three isolated base
points.
The case 
 where $\varphi_{|H_0|}$ is a birational morphism is 
 where the method of \cite{K}
cannot work;
 see also
\cite[Claim 4.9]{O6}.
Applications of our results to the O'Grady conjecture in 
 this case 
when $\varphi_{|H_0|}$ is a birational morphism 
will be discussed in Section \ref{sec-Ogrady}.

\section*{Acknowledgements} We would like to thank K.~O'Grady, L.~Gruson, and
F.O.~Schreyer
for useful discussions. We would also like to thank
 J.~Buczy\'nski, S.~Cynk, I.~Dolgachev, A.~Langer, Ch.~Okonek, and
P.~Pragacz for helpful comments. We thank the referee for useful remarks.
 
\section{Preliminaries} 

It was shown in \cite{Hu} that there 
 are a finite number of
deformation types of hyperk\"{a}hler manifolds with fixed form
$H^2(X,Z)\ni\alpha \mapsto \int \alpha^2 c_2\in \mathbb{Z}$.
In a similar way we obtain the following:

\begin{prop} Let $X$ be an IHS fourfold with $b_2=23$. The Fujiki constant of $X$ is an 
 integer  of
the form $3 n^2$ for some $n\in \mathbb{N}$. In particular the minimal degree of the
self-intersection $H^4$ of an ample divisor $H\subset X$ is $12$ and in this case 
$h^0(\mathcal{O}_X(H))=6$.
\end{prop}

\begin{proof} First from the H-R-R theorem for IHS fourfolds we 
 infer that
\begin{equation}\label{RR}
h^0(\mathcal{O}_X(H))=\chi(\mathcal{O}_X)(H))=\frac{1}
{24} H^4+\frac{1}{24} c_2(X) H^2+\chi(\mathcal{O}_X).
\end{equation}
Next, by the formula of Hitchin and Sawon we
deduce that 
$$
(c_2(X)\cdot \alpha^2)^2=192 \int \sqrt{\hat{A}(X)}\cdot\int\alpha^4
$$ 
for
any class $\alpha\in H^2(X,\mathbb{R})$ where the $\hat{A}$-genus in our case is
just the Todd genus of~$X$.

We claim that $\int \sqrt{\hat{A}(X)}$ is independent of $X$ with $b_2(X)=23$.
Indeed, by the R-R formula as in \cite{HS} we have
$$
\sqrt{\hat{A}}(X)=\frac{1}{2}\hat{A}_2(X)-\frac{1}{8}\hat{A}_1^2(X),
$$ 
where
$\hat{A}_1(X)=\frac{1}{12}c_2$ and $\hat{A}_2(X)=\frac{1}{720}(3c_2^2-c_4)$.
It remains to show that $c_2^2(X)=828$. But this follows from the fact that
$c_4=324$ and $\hat{A}_2=3$. This proves the claim.

We also  deduce that $\frac{(H^2.c_2(X))^2}{H^4}=300$ so $\sqrt{300 H^4}\in
\mathbb{N}$.
It follows that $H^4=3k^2$. On the other hand, from  (\ref{RR}) we deduce
that
$\frac{k^2}{8}+\frac{10k}{8}\in \mathbb{N}$, thus $k$ is even.

Let us now take an element $\alpha\in H^2(X,\mathbb{Z})$ with positive square.
Then there exists a deformation $Y$ of $X$, and $\beta\in H^{1,1}(Y,\mathbb{Z})$ a Gauss--Manin deformation of $\alpha$
such that $\pm \beta$ is ample (Huybrehts projectivity criterion). In particular we infer $\alpha^4=12 m^2$, where
$m\in \mathbb{Z}$ with $\alpha$ of positive square; so also for all $\alpha \in H^2(X,\mathbb{Z})$.
We conclude that the Fujiki constant is of the form $3 n^2$.
\end{proof}

\begin{rem}
For an IHS manifold $X$ with $b_2(X)=23$ to admit an ample divisor with $H^4=12$
there are two possibilities:
\begin{itemize}\item The Fujiki invariant is $3$, the B-B lattice is even
and there exists 
 an $h\in H^2(X,\mathbb{Z})$ with $(h,h)=2$.
\item The Fujiki invariant is $12$ and there 
 exists an $h\in H^2(X,\mathbb{Z})$
with $(h,h)=1$.
\end{itemize}
It is a natural problem to decide whether the latter case can occur.
\end{rem}

\section{The proof of Theorem \ref{EPW}}
The idea of the proof of our theorem is to construct a quadratic symmetric sheaf $\F$ on the unique sextic $S_A$ containing the scheme $C\subset \PP^5$
(we know that the sextic is unique from \cite {K}). We extract $\F$ from a natural resolution of $\varphi_{\ast}(\mathcal{O}_X(2))$.
The first step will be to find a ``symmetric'' resolution of $\varphi_{\ast}(\mathcal{O}_X(2))$.
The second 
 is to restrict this resolution in order to find the equation of the adjoint sextic.

We find that $\varphi\colon X\to X'\subset \mathbb{P}^5=\PP(W)$ is a
birational morphism and a finite map onto a hypersurface of
degree~$12$. Let us consider the Beilinson monad $\mathcal{M}$ applied to
$\varphi_{\ast}(\mathcal{O}_X(2))$.
This is the following complex:
 $$
\dots\to
\bigoplus_{j=0}^{5}H^j(\varphi_{\ast}(\mathcal{O}_X(2+e-j))\otimes
\Omega^{j-e}_{\mathbb{P}^5}(j-e)\to \dots
$$ 
(see \cite{EFS} and \cite{DE}).
We have $H^j(\varphi_{\ast}(\mathcal{O}_X(2-k)))=H^j(\mathcal{O}_X(2-k))$ since
$\varphi$ is finite.
Let us write the monad $\mathcal{M}$ in the following form:
\begin{center}\begin{tabular}{cccccc}
$ H^4(\mathcal{O}_X(-3))$& $ H^4(\mathcal{O}_X(-2))$&$ H^4(\mathcal{O}_X(-1))$&
$\mathbb{C} $&$0$&$0$\\
$0$&$0$&$0$&$0$&$0$&$0$\\
$0$&$0$&$0$&$\mathbb{C} $&$0$&$0$\\
$0$&$0$&$0$&$0$&$0$&$0$  \\
$0$&$0$&$0$&$\mathbb{C} $&$H^0(\mathcal{O}_X(1))$&$ H^0(\mathcal{O}_X(2))$
\end{tabular}
\end{center}
From \cite[Cor.~6.2]{EFS} the 
maps in the last row correspond to the natural
multiplication map $W\otimes H^0(\mathcal{O}(k))\to H^0(\mathcal{O}(k+1))$.
Since by a result of Guan \cite{G} we have
$\Sym^2H^0(\mathcal{O}_X(1))=H^0(\mathcal{O}_X(2))$, the maps in the last row
correspond to the maps in the Beilinson monad of
$\mathcal{O}_{\mathbb{P}^5}(2)$.
Moreover, we denote by $A$ a vector space such that $A^{\vee} \oplus
\Sym^3H^0(\mathcal{O}_X(1))=H^0(\mathcal{O}_X(3))$.
Then analogously the natural complex
$$
0\to \Omega_{\PP^5}^3(3)\to \Omega_{\PP^5}^2(2)\otimes W\to\Omega_{\PP^5}^1(1)\otimes \Sym^2 W\to
\mathcal{O}\otimes \Sym^3W 
$$
is exact and is a free resolution of
$\mathcal{O}_{\mathbb{P}^5}(3)$.
Its Serre dual can be seen as a part of the first row of the monad.

We claim that our Beilinson monad is cohomologous to the following
(cf.~\cite{CS}):
\begin{center}
\begin{tabular}{cccccc}
$\Omega_{\PP^5}^5(5)\otimes A\oplus \oo_{\PP^5}(-4)$ &$0$&$\ \ \ \ \ \ \ 0\ \ \ \ \ \ \ \ $&$\ \ \ \ \ \ \  0\ \ \ \ \ \  $&$\ \ \ \ \ \ 0\ \ \ \ \ \ $&$0$\\
$0$&$0$&$0$&$0$&$0$&$0$\\
$0$&$0$&$0$&$\Omega_{\PP^5}^2(2)$&$0$&$0$\\
$0$&$0$&$0$&$0$&$0$&$0$  \\
$0$&$0$&$0$&$0$&$0$&$\mathcal{O}_{\mathbb{P}^5}(2)$
\end{tabular}
 \end{center}

Let us consider the complex $\mathcal{T}$ constructed from the bottom row of $\mathcal{M}$,
$$
\mathcal{T}\colon 0\to \Omega^2_{\PP^5}(2)  \to \Omega^1_{\PP^5}(1)\otimes  H^0(\mathcal{O}_X(1)) \to  \oo_{\PP^5} \otimes H^0(\mathcal{O}_X(2)) \to 0.
$$
It is naturally a subcomplex of $\mathcal{M}$ such that the quotient
complex is denoted by $\mathcal{M}'$. We have an exact sequence of complexes
\begin{equation}\label{123} 
0\to \mathcal{T}\to \mathcal{M}\to \mathcal{M}'\to 0.
\end{equation}

Denote now by $\mathcal{N}$  the complex obtained by replacing the bottom row of $\mathcal{M}$ by
$\mathcal{O}_{\mathbb{P}^5}(2)$, i.e.
\begin{multline*}
\Omega_{\PP^5}^5(5)\otimes H^4(\oo_X(-3))\to \oo_{\PP^5}(2)\oplus \Omega_{\PP^5}^2(2)\oplus \Omega_{\PP^5}^4(4)\otimes H^4(\oo_X(-2))\\
\to
\Omega_{\PP^5}^3(3)\otimes H^4(\oo_X(-1))\to \Omega_{\PP^5}^2(2) .
\end{multline*}
This complex also maps surjectively onto $\mathcal{M}'$ with kernel $\mathcal{K}$; we thus
obtain another exact sequence of complexes:
\begin{equation}\label{124} 
0\to\mathcal{K}\to \mathcal{N}\to \mathcal{M}'\to 0
\end{equation}
From the long exact  homology sequence associated with  (\ref{123}) we infer that the only non-zero homology spaces are 
$$
H_1(\mathcal{M}')\to \oo_{\PP^5}(2)\to \varphi_{\ast}\oo_{X}(2) \to  H_0(\mathcal{M}').
$$
Looking now at the second sequence (\ref{124}) we infer that the only non-zero homology of $\mathcal{N}$ is in degree $0$.
We deduce from the $5$-lemma that this term is isomorphic to $\varphi_{\ast}\oo_X(2)$.
We treat the upper row of our monad similarly and deduce our claim. 

So from \cite[Thm.~6.1]{EFS} we obtain an exact sequence
\begin{equation}\label{Eq X}
0\rightarrow \mathcal{O}_{\mathbb{P}^5}(-6)\oplus A
\otimes \mathcal{O}_{\mathbb{P}^5}(-3) \xrightarrow{F} \Omega^2_{\mathbb{P}^5}
\oplus\mathcal{O}_{\mathbb{P}^5}\rightarrow \varphi_{\ast}(\mathcal{O}_X
)\rightarrow 0.
\end{equation}
where $A$ is the $10$-dimensional vector space, dual to the quotient
of $H^0(\mathcal{O}_X(3))$ by the cubics of $\mathbb{P}^5$.

We shall show that the sheaf $\varphi_{\ast}(\mathcal{O}_X)$ is symmetric so that we can apply
the results of \cite{EPW} and find that we can choose the map $F$ as symmetric as possible.

 \begin{lemm}
 There exists a symmetric isomorphism 
\[
a\colon
 \varphi_{\ast}(\mathcal{O}_X (3)) \to
\ext^1_{\mathcal{O}_{\mathbb{P}^5}}(\varphi_{\ast}(\mathcal{O}_X
(3)),\mathcal{O}_{\mathbb{P}^5}).
\]
\end{lemm}

\begin{proof}
 By relative duality,
$\HOM_{\mathcal{O}_{\mathbb{P}^5}}(\varphi_{\ast}(\mathcal{O}_X(-3)
),\omega_{X'})=\varphi_{\ast}(\mathcal{O}_X (3))$.
Now applying the functor
$\HOM_{\mathcal{O}_{\mathbb{P}^5}}(\varphi_{\ast}(\mathcal{O}_X
(-3)),\cdot)$ 
to the exact sequence
$$
0\to\mathcal{O}_{\mathbb{P}^5}(-6)\to\mathcal{O}_{\mathbb{P}^5}(6)\to \omega_{X'}\to0
$$
we obtain
\[\HOM_{\mathcal{O}_{\mathbb{P}^5}}(\varphi_{\ast}(\mathcal{O}_X
(-3)),\omega_{X'})\to
\ext^1_{\mathcal{O}_{\mathbb{P}^5}}(\varphi_{\ast}(\mathcal{O}_X
(-3)),\mathcal{O}_{\mathbb{P}^5}(-6))
\xrightarrow{k}\ext^1_{\mathcal{O}_{\mathbb{P}^5}}(\varphi_{\ast}
(\mathcal{O}_X(-3)),\mathcal{O}_{\mathbb{P}^5}(6))
\]
where $k$ is locally given by  multiplication by the equation of
$X'\subset\mathbb{P}^5$, so it is~$0$.
From the projection formula we obtain an isomorphism
 $$
\HOM_{\mathcal{O}_{\mathbb{P}^5}}(\varphi_{\ast}(\mathcal{O}_X
(-3)),\omega_{X'})\to
\ext^1_{\mathcal{O}_{\mathbb{P}^5}}(\varphi_{\ast}(\mathcal{O}_X
(3)),\mathcal{O}_{\mathbb{P}^5}).
$$

To see that $a$ is symmetric we repeat the arguments from \cite[\S 2]{CS}.
First we get 
$$
\Hom_{\oo_{\PP^5}}(\oo_{X'}(3),\oo_{X'}(3))=\Hom_{\oo_{\PP^5}}(\oo_{X'},\oo_{X'})=\mathbb{C}.
$$
Next $a'=\ext^1_{\oo_{\PP^5}}(a, \oo_{\PP^5})=\lambda a$; but
$\ext^1_{\oo_{\PP^5}}(a',\oo_{\PP^5})=a$, thus $\lambda^2=1$, so $\lambda=\pm 1$.
If $\lambda=-1$ then $ \varphi_{\ast}(\mathcal{O}_X (3))$ is skew-symmetric, so arguing as in \cite[\S 2]{CS} we find that the hypersurface $X'\subset \PP^5$ is non-reduced;
this is a contradiction.
It follows that $\lambda=1$, so $a$ is symmetric.
\end{proof}

Since $S^2(\mathcal{O}_{\mathbb{P}^5}(-3)\oplus A \otimes
\mathcal{O}_{\mathbb{P}^5})$ is a sum of line bundles, we deduce that
$$
\Ext_{\mathcal{O}_{\mathbb{P}^5}}^1(S^2(\mathcal{O}_{\mathbb{P}^5}(-3)\oplus A
\otimes
\mathcal{O}_{\mathbb{P}^5}),\mathcal{O}_{\mathbb{P}^5})=0.
$$ 
Thus we
deduce as in the proof of \cite[Thm.~9.2]{EPW} that there is no
obstruction for $a^{-1}$ to be a chain map, so we can find a map $\psi$ that
closes the following diagram:
\begin{equation} \label{diagr} 
\footnotesize
\begin{CD} 
0 @>>> \mathcal{O}_{\mathbb{P}^5}(-3)\oplus \Omega^3_{\mathbb{P}^5}(3) @>F^{\ast}>>
\mathcal{O}_{\mathbb{P}^5}(3)\oplus A^{\vee} \otimes
\mathcal{O}_{\mathbb{P}^5} @>>> \ext^1_{\mathcal{O}_{\mathbb{P}^5}}(\varphi_{\ast}(\mathcal{O}_X
(3)),\mathcal{O}_{\mathbb{P}^5}) @>>> 0\\
@. @V\psi^{\ast}VV @VV\psi V @V a^{-1} VV  @.\\
0 @>>> \mathcal{O}_{\mathbb{P}^5}(-3)\oplus A \otimes
\mathcal{O}_{\mathbb{P}^5} @>F>> \mathcal{O}_{\mathbb{P}^5}(3)\oplus
\Omega^2_{\mathbb{P}^5}(3)@>>> \varphi_{\ast}(\mathcal{O}_X (3)) @>>> 0 
\end{CD}
\end{equation}

Now arguing again as in the proof of \cite[\S 5]{EPW} we can choose a chain map
such that $\psi F^{\ast}$ is a symmetric map.

Our aim now is to make the second step:  extract from $F$ a map $f$ whose determinant gives the adjoint sextic.
We show first that the resolution of the ideal of the conductor of $\varphi$ is obtained by restricting the resolution in (\ref{diagr}).

Recall that the conductor of the finite map $\varphi\colon X\to X'$ is the annihilator of
the $\mathcal{O}_{X'}$-module $\varphi_{\ast}
(\mathcal{O}_X)/\mathcal{O}_{X'}$ and is isomorphic to the
sheaf $\HOM(\varphi_{\ast}(\mathcal{O}_X),\mathcal{O}_{X'})$.
From \cite[7.2 page~249]{H} we deduce that
$\varphi^{!}\omega_{X'}=\omega_{X}$,
so $$
\varphi_{\ast}\oo_X=\varphi_{\ast}(\varphi^{!}\omega_{X'})=\HOM_{\oo_{X'}}(\varphi_{\ast}(\omega_{X}),\oo_{X'}(6)).
$$
On the other hand, $\omega_X$ is trivial, so the conductor $\mathcal{C}$ is isomorphic to $\varphi_{\ast}(\oo_{X}(-6))$.
The inclusion $\mathcal{C}\subset \oo_{X'}$ can be lifted to a map of complexes:
$$
\begin{CD} 
0 @>>> \mathcal{O}_{\mathbb{P}^5}(-12)\oplus A \otimes
\mathcal{O}_{\mathbb{P}^5}(-9) @>F>> \mathcal{O}_{\mathbb{P}^5}(-6)\oplus
\Omega^2_{\mathbb{P}^5}(-6)@>>> \mathcal{C} @>>> 0 \\
@. @VVV @VVV @VVV  @.\\
0 @>>>\oo_{\PP^5}(-12)@>{\det F}>> \oo_{\PP^5}@>>> \oo_{X'}@>>> 0
 \end{CD}
$$

By the mapping cone construction \cite[Prop.~6.15]{E} we obtain the
non-minimal resolution
\[ 
0\to\mathcal{O}_{\mathbb{P}^5}(-12)\oplus A \otimes \mathcal{O}_{\mathbb{P}^5}(-9) \to \mathcal{O}_{\mathbb{P}^5}(-6)\oplus
\Omega^2_{\mathbb{P}^5}(-6)\oplus \oo_{\PP^5}(-12)\to \mathcal{I}_{C|\PP^5} \to 0,
\]
where $C\subset X'\subset\mathbb{P}^5$ is the subscheme defined by the conductor $\mathcal{C}$.
It follows that the following map $b$ is given by  restriction of $F$:
\begin{equation}\label{Eq C} 
0\rightarrow
10\mathcal{O}_{\mathbb{P}^5}(-9)\xrightarrow{b}\Omega_{\mathbb{P}^5}^2(-6)\oplus
\mathcal{O}_{\mathbb{P}^5}(-6)\to\mathcal{I}_{C|\PP^5}\rightarrow 0 .
\end{equation}
Recall that $C$ is supported on the
singular locus of $X'$; moreover, it is locally Cohen--Macaulay has pure
dimension~$3$
and degree~$36$ (see \cite{K}).

Consider the part of $b$ given by 
\begin{equation}\label{part-f} 
A \otimes \mathcal{O}_{\mathbb{P}^5}(-3)
\xrightarrow{f} \Omega^2_{\mathbb{P}^5}. 
\end{equation}
The determinant of this map gives the unique sextic $ S_A\subset
\mathbb{P}^5$ containing $C$. Indeed, taking the long exact sequence associated to  \ref{Eq C} tensorized by $\oo_{\PP^5}(6)$
we see that the unique sextic containing $C\subset \PP^5$ is the image of $H^0(\oo_{\PP^5})\subset H^0(\oo_{\PP^5}\oplus \Omega^2_{\PP^5})$.

Since there is no non-zero map $ \mathcal{O}_{\mathbb{P}^5}(3) \rightarrow
\Omega^2_{\mathbb{P}^5}(3) $, the restriction of the above diagram (\ref{diagr}) gives
$$ 
\begin{CD} \Omega^3_{\mathbb{P}^5}(3) @>f^{\ast}>> A^{\vee} \otimes
\mathcal{O}_{\mathbb{P}^5}\\
@V\rho^{\ast}VV @V\rho VV \\
 A \otimes \mathcal{O}_{\mathbb{P}^5} @>f>>
\Omega^2_{\mathbb{P}^5}(3) 
\end{CD}
$$
where $\rho f^{\ast}$ is a symmetric map which is the restriction of $\psi
F^{\ast}$ to $\Omega^3_{\mathbb{P}^5}(3)$.
We saw in  (\ref{part-f}) that $\det f$ gives the equation of the adjoint sextic.
The cokernels $\F$ and $\F^{\ast}$ of $f$ and $f^{\ast}$ are sheaves supported on the adjoint sextic.
We complete the diagram such that
$$
\begin{CD}
0 @>>> \Omega^3_{\mathbb{P}^5}(3) @>f^{\ast}>> A^{\vee} \otimes
\mathcal{O}_{\mathbb{P}^5} @>>> \F^{\ast} @>>> 0\\
@. @V\rho^{\ast}VV @V\rho VV @V\alpha VV @.\\
0 @>>> A \otimes \mathcal{O}_{\mathbb{P}^5} @>f>>
\Omega^2_{\mathbb{P}^5}(3) @>>> \F @>>> 0 
\end{CD}
$$
 Since $\rho f^{\ast}$ is symmetric, we infer that $\F$ is a symmetric sheaf supported on the adjoint sextic $\det f$ with resolution
$$
0 \to A \otimes \mathcal{O}\xrightarrow{f}
\Omega^2_{\mathbb{P}^5}(3) \to \F \to 0 ,
$$
thus the adjoint sextic is an EPW sextic (see \cite[\S 9.3]{EPW}).

\section{The proof of Proposition \ref{general case}}\label{section-gen} 
For contradiction, suppose that a sextic $S_A$ with $\Theta_A=\emptyset$ can be the adjoint hypersurface which is the image of an IHS fourfold.
Recall that such a generic EPW sextic is singular along a surface of degree $40$ with $A_1$ singularities along this surface.
The idea of the proof of our Proposition is to construct a resolution $V$  of singularities of $S_A$ and then to consider the pull-back of the fourfold $\varphi_{|H|}(X)\subset \PP^5$ on $V$.
We obtain a contradiction by considering the natural duality of $V$. 

We shall first construct the desingularization  $V$ in Section \ref{12}.
In Section \ref{dual} we describe the duality on $V$. The proof of our proposition is given in Section \ref{proof}.

\subsection{Desingularization of EPW sextic}\label{12}
Let us construct $V$. First consider 
$$
O_2=\{[\alpha \wedge \omega] \in \mathbb{P}(\textstyle\bigwedge^3 W) \ | \ \alpha \in W,\
\omega \in\textstyle\bigwedge^2 W\} \subset \PP(\textstyle\bigwedge^3 W),
$$
the closure of the second orbit of the action of $PGL(W)$ on $\PP(\bigwedge^3W)$ (see the Appendix).
Note that $O_2$ is singular along $G(3,W)$; moreover, we have the following diagram:
\begin{equation}\label{diag4}
\xymatrix{ \mathbb{P}(W)& O_2 \ar@{.>}[l]_{\ \ \ \pi_1} \ar@{.>}[r]^{\pi_2\ \ \ \ }&\mathbb{P}(W^{\vee})
}
\end{equation}
such that $\pi_1([\alpha\wedge v])=[\alpha]\in \PP(W)$ and $\pi_2([\alpha\wedge v])=[\alpha\wedge v\wedge v]\in \PP(W^{\vee})$ for $\alpha \in W,\
\omega \in\bigwedge^2 W$. The maps $\pi_1$ and $\pi_2$ are rational, and  defined outside $G(3,W)\subset O_2$ by Lemma \ref{well-defined}.

Now the wedge product $\bigwedge^3W\oplus
\bigwedge^3W\to\bigwedge^6W=\mathbb{C}$ induces a skew-symmetric form on $\bigwedge^3 W$. We consider a maximal $10$-dimensional Lagrangian subspace $A\subset \bigwedge^3 W$ isotropic with respect to this form.
For a fixed $A$ we define the manifold 
$$V':=\mathbb{P}(A)\cap O_2.
$$

\begin{prop}\label{smoothV} 
The image $\pi_1(V')$ is the EPW sextic $S_A$ associated to $A$.
Moreover, if $\Theta_A=\emptyset$ then $V'$ is a smooth Calabi--Yau fourfold.
\end{prop}

\begin{proof} In order to find the image $\pi_1(V')$ we consider a natural desingularization of $O_2$ which is the projectivization of the vector bundle $\PP(\Omega_{\PP^5}^3(3))$. Comparing the following construction with the definition of the EPW sextic given in the introduction (see (\ref{eqEPW})) we deduce the first part of the statement.

Let us describe this desingularization.
From \cite[Thm.~9.2]{EPW} we can see that ${f^{\ast}}\choose{\rho^{\vee}}$ defines an
embedding of $\Omega_{\PP^5}^3(3)$ as a symplectic subbundle of $(A\oplus
A^{\vee})\otimes \mathcal{O}_{\mathbb{P}^5}=\bigwedge^3W \otimes
\mathcal{O}_{\mathbb{P}^5}$.
On the other hand, from \cite[\S 5.2]{O1} we deduce that we can look at
$\bigwedge^3W \otimes \mathcal{O}_{\mathbb{P}^5}$ as a symplectic vector bundle
with the symplectic form induced from the wedge product $\bigwedge^3W\oplus
\bigwedge^3W\to\bigwedge^6W=\mathbb{C}$
such that the fiber of the subbundle $\Omega_{\PP^5}^3(3)$ over $v\in
\mathbb{P}^5$ corresponds to the $10$-dimensional linear space
$$
F_v=\{[v \wedge \gamma]\in \PP(\textstyle\bigwedge^3W)\mid \gamma \in  \textstyle\bigwedge^2W \} \subset
\PP(\textstyle\bigwedge^3W).$$
Then $\rho^{\ast}$ is given by the above embedding composed with the quotient map
$$\textstyle\bigwedge^3W\otimes \mathcal{O}_{\mathbb{P}(W)}\to (\textstyle\bigwedge^3W / A)\otimes
\mathcal{O}_{\mathbb{P}(W)},$$
where $A$ is a Lagrangian subspace of $\bigwedge^3W$ (there is a canonical
isomorphism $\bigwedge^3W / A=A^{\vee}$).
More precisely, we have a diagram
$$
\xymatrix{\PP(\textstyle\bigwedge^3 W)\supset O_2 \ar@{.>}[r]^{\ \ \ \ \pi_1} & \PP^5\\
& \PP(\Omega_{\PP^5}^3(3)) \ar[u]^{\pi} \ar[ul]^{\alpha}}$$
 such that the image of $\alpha$ is the variety $O_2$. The first part follows.

Let us prove the second part. The smoothness follows from Proposition \ref{tgW_2}. Indeed, suppose that $V'$ is singular at a point $p$.
Then $\PP(A)$ intersects the tangent space to $O_2$ in a non-transversal way along a $5$-dimensional isotropic subspace $Z$.
By Proposition \ref{tgW_2} the space $Z$ has to cut $G(3,W)$, a contradiction.

Let us find the canonical bundle of $V'$.
Observe that $\alpha$ is given by the complete linear system of the \emph{line bundle}
$\mathcal{T}:=\mathcal{O}_{\mathbb{P}(\Omega^3_{\mathbb{P}^5}(3))}(-1)$.
Denote  
$$
V:=\alpha^{-1}(\mathbb{P}(A)\cap O_2).
$$
Since $V$ is smooth, $V'$ is isomorphic to $V$.
We find the canonical divisor of $V$ using the adjunction formula and the knowledge of the canonical divisor of $\PP(\Omega^3_{\PP^5}(3))$.
The dimension of the cohomology group 
$h^1(\oo_{V'})$ is found by using the Lefschetz hyperplane theorem.
\end{proof}
\begin{rem}\label{rema}
 Alternatively, for the proof of the last Proposition, we can use the results from \cite{O1} to prove that $V$ is the blow-up of the quotient of $X_A$ by an anty-symplectic involution. We infer in this way that $V$ is a smooth Calabi--Yau fourfold with Picard group of rank $2.$
\end{rem}

\begin{rem}\label{alpha}
Denote by $E$ the exceptional divisor of $\alpha$. It maps to $G(3,W)\subset
O_2$ such that the fiber over a
point $U\in G(3,W)$ is a projective plane that maps under $\pi$ to
$\mathbb{P}(U)\subset\mathbb{P}(W)$.
Moreover, $E$ is isomorphic to the projectivization of the tautological bundle
on $G(3,W)$.
By Lemma \ref{H_1+H_2} we deduce that the pull-back $(\alpha\circ
\pi_2)^{\ast}(H_2)$ is a Cartier divisor in the linear system $|2T-H|$ on
$\mathbb{P}(\Omega_{\PP^5}^{3}(3))$. Moreover, $E$ is the base locus of $|2T-H|$ such that
after blowing-up $E\subset \mathbb{P}(\Omega_{\PP^5}^3(3))$ this linear system
become base-point-free and factorizes through $\pi_2$.
\end{rem}

The idea of the proof of Proposition \ref{general case} is by contradiction.
Denote by $H$ the pull-back by $\pi \colon V\to \PP(W)$ of the hyperplane section in $\PP(W)$ and $T\in |T|$. From Proposition \ref{Pic(W_2)} and the Lefshetz theorem (or from Remark \ref{rema}) the divisors $H$ and $T$ generate $\Pic(V)$.
First we need the following:

\begin{prop}\label{ciC}  There exists a divisor $D\subset V$ in the linear system $|3H+T|$ that
projects under $\pi$ to $C$.
\end{prop}
\begin{proof}

We shall show that $D$ is given by the vanishing of a section of the
vector bundle 
$10 \mathcal{O}_{\mathbb{P}(\Omega^3_{\mathbb{P}^5}(3))}(-1)\oplus
(\mathcal{O}_{\mathbb{P}(\Omega^3_{\mathbb{P}^5}(3))}(1)\otimes
\mathcal{O}_{\mathbb{P}^5}(3))$
on $\mathbb{P}(\Omega^3_{\mathbb{P}^5}(3))$.
 Recall that the sequence (\ref{Eq C}) defines a codimension~$1$ subscheme
$C\subset S_A$. Let us apply Kempf's idea and pull back $b^{\vee}$ (where $b$ is
defined by  (\ref{Eq C})) by $p\colon \mathbb{P}(10\mathcal{O}_{\mathbb{P}^5})\to \PP^5$.
Then as in \cite[Appendix B]{L} we obtain a diagram
\begin{equation}\label{diag3}
\xymatrix{p^{\ast}\Omega_{\mathbb{P}^5}^3(3)\oplus
p^{\ast}\mathcal{O}_{\mathbb{P}^5}(-3) \ar[r]^{\ \ \ \ \ \ \ \ \ \ p^{\ast}b^{\vee}} \ar[rd]^v&
p^{\ast} (10\mathcal{O}_{\mathbb{P}^5}) \ar[d]\\
& \mathcal{O}_{\mathbb{P}(10\mathcal{O}_{\mathbb{P}^5})}(1)}
\end{equation}

We see that the degeneracy locus of $b^{\vee}$ can be seen on
$\mathbb{P}(10\mathcal{O}_{\mathbb{P}^5})$ as the degeneracy of $v$, thus as 
the
zero section of
$$(\mathcal{O}_{(10\mathcal{O}_{\mathbb{P}^5})}(1)\otimes
p^{\ast}\mathcal{O}_{\mathbb{P}^5}(3))\oplus
(\mathcal{O}_{(10\mathcal{O}_{\mathbb{P}^5})}(1)\otimes p^{\ast}
\Omega_{\PP^5}^2(3)).$$ Finally, note that the zero scheme of the bundle
$\mathcal{O}_{(10\mathcal{O}_{\mathbb{P}^5})}(1)\otimes
p^{\ast}(\Omega_{\mathbb{P}^5}^2(3))$ defines $V$ set-theoretically, and the
 restrictions
 $\mathcal{O}_{(10\mathcal{O}_{\mathbb{P}^5})}(1)|_V$ and $\mathcal{O}_{\PP(\Omega^3_{
\mathbb{P}^5}(3))}(-1)|_V$ are equal.
\end{proof}

Finally, we shall translate  geometrical properties of the map $\varphi\colon
X\to X'\subset \PP^5$ into geometrical properties of the adjoint EPW sextic.
Let us also consider the subschemes $N_r\subset X'$ defined by 
$Fitt_{r}^{X'}(\varphi_{\ast}(\mathcal{O}_X))$ (for example
$N_1=C$). Recall that from the results of \cite[\S 4]{MP} the scheme
$N_2$ has a symmetric presentation matrix and is of codimension
$\leq3$ if it is non-empty. Moreover $N_2$ is supported on points
where $C$ is not a locally complete intersection (see
\cite[p.~131]{MP}).
Denote by $M_r$ the degeneracy 
locus of rank $\leq 10-r$ of the map 
\[ 
A \otimes
\mathcal{O}_{\mathbb{P}^5}(-3) \xrightarrow{f} \Omega^2_{\mathbb{P}^5}.
\]

\begin{lemm}\label{N3 5.1}
 The subschemes $N_2$ and $M_2$ of $\mathbb{P}^5$ are equal and the
radicals of the schemes $N_r$ and $M_r$ are equal for $r\geq 2$. Moreover, suppose that $p\in M_{k}-M_{k+1}$. Then for $k\geq 1$ the dimension of the intersection
$F_p\cap\mathbb{P}(A)$ is $k-1$.
\end{lemm}

\begin{proof} This is an analogous statement to the rank condition
(see \cite[Rem.~2.8]{CS}). We claim that locally the map $F$ can be seen as a
symmetric map.
Indeed, in the diagram (\ref{diagr}) using alternating homotopies as in
\cite[p.~447]{EPW} we have the freedom of choice of the map 
$\psi$.
In particular restricting to an affine neighborhood we can assume that the
matrix $A:=F \psi$ is symmetric and that $\psi$ is an isomorphism.
Note that the matrix $B$ consisting of the last $9$ columns of $A$ and the
matrix $B'$ which is the last $9$ rows of $B$ have maximal degeneracy loci
defining locally the scheme $C$ and the sextic $S_A$ respectively (see  (\ref{Eq C})). Since we know that $X'$ has a non-singular normalization,
we can conclude with \cite[Prop.~3.6(3)]{KU}. 

For the second part we use
\cite[Lem.~2.8]{KU}.
 It follows from the proof of Proposition \ref{ciC} that the dimension of the fiber $V\cap
\pi^{-1}(p)$ is equal to $k-1$.
We conclude by observing that the map $\alpha$ does not contract curves on
$\pi^{-1}(p)$.
\end{proof}


\subsection{The duality}\label{dual}

Since we have a second fibration $\pi_2$ of the variety $O_2$,
it is natural to consider the following picture:
 \begin{equation}\label{diag4'}
\xymatrix{  \mathbb{P}(W)& O_2\subset\mathbb{P}(\textstyle\bigwedge^3
W) \ar@{.>}[l]_{\pi_1} \ar@{.>}[r]^{\ \ \ \ \ \pi_2}&\mathbb{P}(W^{\vee}) \\
 \mathbb{P}(\Omega^3_{\mathbb{P}(W)}(3))\ar[u]^{\pi} \ar[ru]_{\alpha} &&\mathbb{P}(\Omega^3_{\mathbb{P}(W^{
\vee})}(3)) \ar[u]^{\pi'} \ar[lu]^{\alpha'}}
\end{equation}
Denote by $F_v'$ the closure of the fiber of $\pi_2$ and by $\pi_2(V')=S_A'\subset \mathbb{P}(W^{\vee})$
the corresponding EPW sextic constructed from~$A$.
Denote by $\mathcal{O}_{V'}(H_2):=\pi_2^{\ast}(\mathcal{O}_{\mathbb{P}(W^{\vee})}(1))$.
Without lost of generality we can denote $\mathcal{O}_{V'}(H):=\pi_1^{\ast}(\mathcal{O}_{\mathbb{P}(W)}(1))$ (we identify it with the divisor $\pi^{\ast}(\oo_{\PP(W)}(1))$ on $V$).
\begin{lemm}
 Assume that the sextic $S_A\subset\mathbb{P}(W)$ is integral. Then
$S_A'\subset\mathbb{P}(W^{\vee})$ is integral and dual to~$S_A$.
\end{lemm}

\begin{proof}
It follows from the definition of $\pi_1$ and $\pi_2$ that $\pi_2(F_v)$ is a
hyperplane in $\mathbb{P}(W^{\vee})$ that is dual to $v\in\mathbb{P}(W)$.
Next it follows from the description in \cite[Cor.~1.5(2)]{O2} of the tangent
space $T$ to $S_A$ at a smooth point that there is a point $w\in S'_A$ such that
$\pi_1(F_w')=T$.
\end{proof}
\begin{rem}\label{F_v}
As remarked by O'Grady \cite[\S 1.3]{O2}, the map $\pi_2|_{F_v}$ is given by
the linear system of Pl\"{u}cker quadrics defining  $F_v\cap
G(3,W)=G(2,5)\subset\mathbb{P}^9$.
Thus the fibers of $\pi_2|_{F_v}$ are $5$-dimensional linear spaces spanned by
$G(2,4)\subset G(2,5)\subset \mathbb{P}^9$.
\end{rem}

\subsection{The proof of (\ref{general case})}\label{proof}
The aim of this section is to prove that an EPW
sextic $S_{A}$ constructed by choosing $\mathbb{P}(A)$ disjoint from $G(3,W)$ (i.e.~with $\Theta_A=\emptyset$) cannot be the adjoint hypersurface of a birational
image of an IHS manifold with $b_2=23$.

For contradiction, suppose that  $S_A$ can be such a hypersurface.
Then for the corresponding Lagrangian space $A$ with $\Theta_A=\emptyset$ the variety  
$V'=O_2\cap\mathbb{P}(A)$ is isomorphic to $V=\alpha^{-1}(V')$.
Thus let us identify $V=V'$. 

From \cite[Claim 3.7]{O3} we deduce that there are only a finite number of
planes on $V$ contracted to points by $\pi_1$ and there are no higher dimensional
contracted linear spaces.
Denote by $E$ and $E_2$ the exceptional loci of $\pi_1|_{V}$ and $\pi_2|_{V}$
respectively, by $T$ the restriction of the hyperplane in $\mathbb{P}(\bigwedge^3
W)$, and by $H$ and $H_2$ the pull-backs by $\pi_1$ and $\pi_2$ respectively of
the hyperplane sections in $\PP(W)$ and $\PP(W^{\vee})$.
By \cite[Prop.~1.9]{O2} the singular locus of $S_A$ is a surface
$G$ of degree $40$ that is smooth outside  the image of the contracted planes.
Moreover, $S_A$ has ODP singularities along the smooth locus of $G$. Hence $E$
and $E_2$ are reduced.

Using Proposition \ref{Pic(W_2)} and the Lefschetz theorem \cite[Thm.~1]{RS}, which works when $V$ is smooth and omits the singular locus of $O_2$, we deduce that the
Picard group of $V$ has rank $2$ and is generated by the restrictions of $H$
and $H_2$.

\begin{lemm}\label{W_2}
In $CH^1(V)=\Pic(V)$ we have the  equalities  
$$H_2=5H-E\ \ \text{and}\ \  H+H_2=2T.
$$
\end{lemm}
\begin{proof}
The first equality  follows from \cite[\S 1.2.2]{D} (see Corrolary \ref{dol}) and the second from Lemma
\ref{H_1+H_2}.
 \end{proof}
 
Now, using Proposition \ref{ciC} we find a divisor $D\subset V$ in the linear
system $|3H+T|$ such that $p(D)=C$.
It follows from Lemma \ref{N3 5.1} that $D-E$ is an effective divisor $D_1$.
 Let $l\subset V$ be a line contracted by $\pi_2$ (such lines cover $E_2$).
Since $l.T=1$, from 
Lemma \ref{W_2} we obtain $l.H=2$. It follows
that
$$l.(D-E)=l.(3H+T-E)=l.(T+H_2-2H)=-3$$ in $CH^1(V).$
Since $D_1$ is effective, we infer that $l\subset D_1$, thus $E_2\subset D_1$ and
$D_1-E_2$ is effective ($E_2$ is reduced).
We obtain the following equalities  in $CH^1(V)$:
 $$
D_1-E_2=T-4H_2-H=-3H_2-T.
$$
This is a contradiction because $-3H_2-T$ cannot be effective.


 \section{On the O'Grady conjecture}\label{sec-Ogrady}
 The aim of this section is to apply the results from the previous sections to prove some special cases of  Conjecture \ref{OGcon} of O'Grady.
 In fact, we shall generalize the results of Proposition \ref{general case} to a special class of IHS fourfolds
with $b_2=23$ satisfying an additional condition \textbf{O} described in the next subsection.

\subsection{IHS fourfolds with $b_2=23$ satisfying condition \textbf{O}}
Let  $(X,H)$ be a polarized IHS fourfold with $b_2=23$ such that $H^4=12$. Consider the following definition:

\begin{defi}
We say that $(X,H)$ satisfies \textit{condition} \textbf{O} if for all $D_1,D_2,D_3\in |H|$ that are independent, the intersection $D_1\cap D_2\cap D_3$ is a curve.
\end{defi}

Intuitively,  condition \textbf{O} says that the image $\varphi_{|H|}(X)\subset \PP^5$ does not contain planes.
Note that this is one of the conditions from \cite[Claim 4.4]{O6}. Moreover, each
IHS manifold numerically equivalent to  $Hilb^2(S)$, where $S$ is a $K3$ surface, can be deformed
to one that satisfies condition \textbf{O}. Motivated by this we can state the following:

\begin{prob}
Is each IHS fourfold with $b_2=23$ deformation equivalent to a polarized IHS fourfold $(X_0,H_0)$ satisfying condition \textbf{O} such that $H_0^4=12$?
\end{prob}

Note that if we find such a deformation, we can repeat the arguments from \cite{O} in order to show
 that  either $\varphi_{|H_0| }$ is the double cover of an EPW sextic (thus $X_0$ is of type $K3^{[2]}$)
 or $X_0$ is birational to a hypersurface of degree $12\geq d\geq 7$, or a $4:1$ morphism to a cubic hypersurface with isolated singularities, or a $3$-to-$1$ morphism to a normal quartic hypersurface, or $\dim \varphi_{|H_0|}(X_0)\leq 3$. It is a natural geometric problem to decide which one of the above cases can occur.

\subsection{The O'Grady conjecture}\label{secOG}
In this section we discuss the following conjecture of O'Grady:

\begin{conj}\label{OGcon} If an IHS fourfold $X$ is numerically
equivalent to  $Hilb^2(S)$ where $S$ is a $K3$ surface (i.e.~$c=3$ and $(H^2(X,\mathbb{Z}),q)$ is
isometric to $U^3\oplus E_8^2\oplus \langle -2\rangle$ with the standard
notation), then it is deformation equivalent to it.
\end{conj}

Let $X$ be an IHS manifold numerically equivalent to
$S^{[2]}$ where $S$ is a $K3$ surface. Consider $\mathcal{M}_X'$, a
connected component of the moduli space of marked IHS
fourfolds deformation equivalent to $X$, and the surjective period
map 
$$P\colon\mathcal{M}_X'\to \Omega_L.
$$ 
Then choose an
appropriate $\rho\in \Omega_L$ such that $P^{-1}(\rho)$ is an
IHS manifold $X$ deformation equivalent to $X_0$ and
$\Pic(X_0)=\mathbb{Z}H_0$ where $H_0$ is an ample divisor with $H_0^4=12$. The
special choice of $\rho$ requires $X_0$ to satisfy condition \textbf{O} and additional conditions that are
described in
\cite[Claim 4.4]{O6}. For such $(X_0,H_0)$ O'Grady proved that the linear
system $|H_0|$ gives a map $\varphi_{|H_0|}$ of degree $\leq 2$ that is either
birational onto its image or a special double cover of an EPW sextic.
Since this double cover is deformation equivalent  to $Hilb^2(S)$ where $S$ is a $K3$ surface,
his conjecture follows if we prove that $\deg \varphi_{|H_0|} \neq 1$. 

If we suppose
that $\deg\varphi_{|H_0|}=1$ (i.e $\varphi_{|H_0|}$ is a birational map) then O'Grady remarked that the image of
$\varphi_{|H_0|}$ is a hypersurface of degree $6\leq d\leq 12$. In \cite{K}
we showed that $d\geq 9$ and  $|H_0|$ has at most three isolated base
points. Note that if $|H_0|$ has one isolated point, the scheme defined by the ideal of the conductor
of $\varphi_{|H_0|}$ is contained in a unique quintic (containing the singular locus of $\varphi_{|H_0|}$). 
There is a lot of geometry appearing as discussed in \cite{Gr}.

In this work we consider the case $d=12$ (i.e.~$|H_0|$ has no base points); this is the case where the
method of \cite{K} does not work and also the most difficult one from the point of
view of
O'Grady (see \cite[Claim 4.9]{O6}). Then the image of $\varphi_{|H_0|}$ is a non-normal
degree $12$ hypersurface
$\varphi_{|H_0|}(X')\subset \mathbb{P}(W)$. 
Our idea is to study the adjoint
hypersurface $S_A$ to $X'\subset \mathbb{P}(W)$. We know that it is an EPW sextic so we can use the classification of such sextics given in \cite{O2}, \cite{O3},
\cite{O4} and \cite{IM} in order to describe $S_A$ more precisely. 
Recall that for $S_A$ the set $\Theta_A$ (defined in (\ref{2.2})) is empty for a generic choice of $A$ and if $\Theta_A\neq \emptyset$ it measures how
singular the EPW sextic is. For special $A$ all the values $0\leq \dim
\Theta_A\leq 6$ can be obtained.

Recall again that each numerical $(K3)^{[2]}$ can be deformed to a polarized IHS fourfold $(X_0,H_0)$ that satisfies condition \textbf{O}. Our main result of this section is  the following:
\begin{prop}\label{og-conj}
Suppose that a hypersurface $X'\subset \mathbb{P}^5$ of degree $12$ is the
birational image of a polarized IHS manifold $(X,H)$ with $b_2=23$  such that $H^4=12$ satisfying \textbf{O} through a morphism given by the complete linear system $|H|$. Let
$S_A\subset \mathbb{P}^5$ be the adjoint EPW sextic to the image $X'\subset \PP^5$. Then 
for $S_A$ we have 
either $\dim \Theta_A=1$, or $S_A$ is the double determinantal
cubic, or $S_A$ has a non-reduced linear component.
\end{prop}

The idea of the proof is as follows:
We separately treat the cases when $\dim \Theta_A=0$, $\dim \Theta_A\geq 2$ and $\dim \Theta_A=1$
(in Sections \ref{Theta0}, \ref{Theta2} and \ref{Theta1} respectively).
The case $\dim \Theta_A=0$ is similar to $ \Theta_A=\emptyset$.
In the other cases for each point $U\in \Theta_A$ we consider the plane $\mathbb{P}(U)\subset\mathbb{P}^5$
 contained in $S_A$ such that $S_A$ is singular along $\PP(U)$.
Then we consider after O'Grady (see \cite{O4}) the sets $\C_{U,A}\subset
\mathbb{P}(U)$ defined in (\ref{CU}) below. Each $\C_{U,A}\subset
\mathbb{P}(U)$ is either the whole plane or  the support of some sextic curve $C_{U,A}$.
We show that $\C_{U,A}$ has to be contained in $X'$, thus cannot be a plane (by
 condition \textbf{O}). We also show that  $\C_{U,A}$  must have
degree $\leq 3$ (see Lemma \ref{cubic}).
Checking case by case we exclude all the possibilities with $\dim \Theta_A\geq
2$ except when either $S_A$ is the double determinantal cubic and $X'$ has
generically tacnodes along  $S_A\cap X'$, or $S_A$ is reducible
and equal to $2H_0+Q$ where $H_0$ is a hyperplane and $Q$ a quartic such that
$H_0\cap Q$
supports the scheme $C$ defined by the conductor. In particular in the second case $X'$ has
triple points along $C$ that are not ordinary triple points (see the end of
Section \ref{Theta2} for a precise description). A new idea is needed to
conclude in those cases.

We believe that by the methods of this paper we can also exclude the case $\dim
\Theta_A=1$, but the problem becomes more technical and we only show that the Lagrangian subspace $A\subset \bigwedge^3 W$ defining $S_A$ cannot be generic in the set
of Lagrangian $A$ with $\dim \Theta_A=1$ (see Section
\ref{Theta1}). 
Before proving Proposition \ref{og-conj}, we first introduce some technical results.
\subsection{Preliminary results}

For $U\in G(3,W)$ we see that $\pi(\alpha^{-1}(U))= \mathbb{P}(U)\subset
\mathbb{P}(W)$ is the corresponding plane contained in $S_A$.
Let us consider after O'Grady the set 
\begin{equation}\label{CU}
\mathcal{C}_{U,A}:=\{[v]\in
\mathbb{P}(U)\mid \dim(F_v\cap \mathbb{P}(A))\geq 1 \},
\end{equation}
where $F_v$ is the linear space being the closure of the fiber of the map $\pi_1\colon O_2\dasharrow \PP(W)$ at the point $[v]$. 
There is a natural scheme structure $C_{U,A}$ on $\mathcal{C}_{U,A}$ described in \cite[\S
3.1]{O4} such that $C_{U,A}$ is either a sextic curve or the whole plane
$\mathbb{P}(U)$.

\begin{prop}\label{propC_(U,A)}
The set $\C_{U,A}$ is contained in $ X'\subset \mathbb{P}(W)$. In particular
$\C_{U,A}$ is never equal to $\mathbb{P}(U)$
if $(X,H)$ satisfies condition \textbf{O}.
\end{prop}

\begin{proof}
First, over the points from the set $\C_{U,A}$ the map  
\[ 
A \otimes
\mathcal{O}_{\mathbb{P}^5}(-3) \xrightarrow{f} \Omega^2_{\mathbb{P}^5}
\]
has corank $\geq 2$; so $\C_{U,A}\subset M_2$. But from Lemma
\ref{N3 5.1} we have $N_2=M_2$, thus 
$$
C_{U,A}\subset M_2=N_2\subset X'.
$$
Finally, it follows from condition \textbf{O} that $X'\subset\mathbb{P}(W)$
cannot contain any plane.
\end{proof}

\begin{defi}Recall that O'Grady defined, for $A\in \mathbb{LG}(10,\bigwedge^3W)$ and $U\in
\Theta_A$, the set 
$\mathcal{B}(U,A)$ of $v\in\PP(U)$ such that
either
\begin{enumerate}
 \item there exists $U'\in (\Theta_A - \{U \})$ such that $v\in
\mathbb{P}(U')$, or
\item $\dim(\mathbb{P}(A) \cap F_v \cap T_U )\geq 1,$
\end{enumerate}
where $T_U$ is the projective tangent space to $G(3,W)$ at $U$.
\end{defi}

\begin{lemm}\label{property-SingC_U} The curve
$C_{U,A}\subset \PP(U)$ can have only isolated singularities outside $\mathcal{B}(U,A)$.
If $\mathbb{P}(U)\neq \C_{U,A}$ then  $\mathcal{B}(U,A)\subset\sing
C_{U,A}$.
Moreover, if $U_1, U_2\in \Theta_A$ then $\mathbb{P}(U_1)$ and $\mathbb{P}(U_2)$
intersect as planes in
$\mathbb{P}^5$ at the 
 point of intersection $\C_{U_1,A}\cap\C_{U_2,A}$.
\end{lemm}

\begin{proof} 
This is proved in \cite[Cor.~3.2.7]{O4}. 
\end{proof}

 We have the following description of the EPW sextic $S_A$ in the case $\dim \Theta_A=0$.

\begin{lemm}
If $\dim \Theta_A=0$ then $S_A$ is normal. Moreover, $V=\alpha^{-1}(V')$ and $V'= \PP(A)\cap O_2$
are irreducible.
\end{lemm}

\begin{proof} Since $S_A$ is locally a complete intersection, the normality of
$S_A$ follows from the Serre criterion
 if $S_A$ is non-singular in codimension
$1$.
On the other hand, it follows from \cite[\S 1.3]{O2} that $S_A$ is only singular along the sum of the
planes $\mathbb{P}(U)$ for $U\in \Theta_A$ and along the set $\mathcal{D}$ such that for
$v\in \mathcal{D}$ we have
$$ 
F_v\cap \mathbb{P}(A)\cap G(3,W)=\emptyset \ \ \text{and} \ \
\dim (F_v\cap \mathbb{P}(A))\geq 1.
$$ 
From \cite[Prop.~1.9]{O2} we infer that
$\mathcal{D}$ is a surface.

Since the intersection of $\PP(A)$ with the tangent to $O_2$ at $P$ is
5-dimensional isotropic, we deduce from Proposition \ref{tgW_2} that  $\mathbb{P}(A)\cap O_2$ is smooth at $$P\in (F_v\cap
\mathbb{P}(A))-G(3,W)$$ when $F_v\cap \PP(A)\cap G(3,W)=\emptyset$. Thus we have
to show that the dimension of the exceptional set of $\pi\colon V\to S_A$ that
maps to $\G:=(\bigcup_{U\in \Theta_A}\PP(U))_{red}$ is smaller than $4$.
From the fact that $\Theta_A$ is a finite set it is enough to consider the exceptional set
above $\C_{U_0,A}\subset \PP(U_0)$ for a fixed $U_0\in\Theta_A$.
Since $\Theta_A$ is finite, the fiber $\alpha(\pi^{-1}(v))\subset F_v$ for a
given $v\in \C_{U_0,A}$ intersects $G(2,5)=G(3,W)\cap F_v$ in a finite number of points.
Since the dimension of $G(3,5)\subset\PP^9$ is $6$, we infer $\dim
\pi^{-1}(v)\leq 3$ for all $v\in\C_{U_0,A}$ and  $\dim \pi^{-1}(v)\leq 2$
for a generic $v\in\C_{U_0,A}$.
It follows that $V'$ and $V$ are irreducible.
\end{proof}

The map $V\xrightarrow{\alpha} V'=\PP(A)\cap O_2$ is an isomorphism outside
$\alpha^{-1}(G(3,W))$.
Thus from the proof above we deduce that if $\dim \Theta_A=0$ then $V$ can only be singular at points that
map to a curve $C_{U,A}$ for some $U\in \Theta_A$.

\begin{prop}\label{normalV}
 If $\dim \Theta_A=0$ the varieties $V=\alpha^{-1}(V')$ and $V'=\PP(A)\cap O_2$ are non-singular in codimension $1$.
Moreover, $V$ is normal.
\end{prop}

\begin{proof}
 Note that $V$ is locally a complete intersection, thus it is enough to show the
first part.
 Our aim is to show that the singular points of $\PP(A)\cap O_2$ are contained
in the sum of the tangent spaces to $G(3,W)$ at points from $\Theta_A$.
 Next we show that the intersection of $\PP(A)\cap O_2$ with those tangent
spaces is of codimension~$2$.

We need to consider points in the pre-image 
$$
B':=\pi^{-1}(\C_{U,A}).
$$
Denote by $B$ an irreducible component of $B'$.
Suppose that for a given $U_0$ this set is $3$-dimensional; then either there is
a one-parameter family of planes parameterized by $\C_{U_0,A}$ or there is a
three-dimensional linear space (i.e. $\mathbb{P}^3$) mapping to a point on $\C_{U_0,A}$. Let us consider the first
case; the other is treated similarly.

Suppose that $V$ is singular along $B$. Then  at each  $p\in
\alpha(B)-G(3,W)$ the space $\mathbb{P}(A)$
does not intersect transversally the tangent plane to $O_2$. By Proposition \ref{tgW_2} the intersection
$\mathbb{P}(A)\cap F_v\cap F'_w$, where $v=\pi_1(p)$ and $w=\pi_2(p)$, contains
the line $[p,U']\subset F_v$ where $U'$ is one of the finite number of points (at most two) in
 $\mathbb{P}(A)\cap G(3,W)\cap F_v\cap F'_w$.

We claim that for a generic choice of $p\in B$ the line $[p,U']$ is contained in
the tangent space to $G(2,5)\subset F_v$ at $U$ (i.e. $B\subset T_U$).
Since $\Theta_A$ is finite,  for a generic choice of $p\in B$ the
line $[p,U']$ with $U'\in \mathbb{P}(A)\cap G(3,W)$ intersects $G(3,W)$ in one
point.
From Remark \ref{F_v} the line is contained in a five-dimensional linear space
$L_p=F_v\cap F'_w$ such that $L_p\cap G(2,5)$ is a quadric.
Since this line intersects $G(3,W)$ in one point, it has to be tangent to $G(2,5)$.
The claim follows.

Let $T_U$ be the projective tangent space to $U\in G(3,W)$ and
$M_U:=\mathbb{P}(A)\cap T_U$.
The following is a nice exercise.

\begin{lemm}\label{lemP2xP2}
 The intersection $K_U:=T_U\cap G(3,W)$ can be seen as 
the set of planes in
$\mathbb{P}(W)$ that intersect the plane $\mathbb{P}(U)$ along a line.
In particular  $K_U$ has dimension~$5$ and is a cone
over the Segre embedding of $\mathbb{P}^2\times\mathbb{P}^2$.
The sum of the linear spaces $F_v\cap
T_U$ for $v\in \mathbb{P}(U)$ is a cone over the determinantal
cubic $E_U$.  Moreover, $K_U$  
is the singular set of $E_U$.
\end{lemm}

First we have $\dim M_U\leq 4$, since otherwise we infer 
$$\dim \Theta_A\geq \dim
(M_U\cap K_U)\geq 1,
$$  
contrary to $\dim \Theta_A=0$.
So we have three possibilities: $\dim M_U=2$, $3$, or $4$.
Note that  $F_v\cap T_U$ is the tangent space to $G(2,5)$ at $U$,
so has dimension~$6$.
If $\dim M_U=4$ then each linear space $F_v\cap T_U$ for $v\in
\mathbb{P}(U)$ intersects $\PP(A)$ along a linear space of dimension at least
$1$ (because such an intersection contains  $F_v\cap M_U$). Thus $\C_{U,A}=\PP(U)$, contrary to Proposition \ref{propC_(U,A)}.

So we can assume that $\dim M_U\leq 3$.
We saw above that the generic fiber of $\pi|_B\colon B\to C_{U,A}$ is a plane contained in
$T_U$. Since these
fibers are contained in $M_U$ and disjoint outside $U$, we obtain a
contradiction.
\end{proof}
Finally, we will use several times the following:

\begin{prop}\label{dual-tangent}
Suppose that the set of points $v\in\PP(W)$ with $\dim (F_v\cap \PP(A))\geq 2$ is a
curve $C\subset \PP(W)$. Then the tangent space $T_{v_0}\subset \PP(W)$ to $C$ at $v_0$ is perpendicular
to the linear space spanned by the image $\pi_2(\PP(A)\cap F_v)\subset \PP(W^{\vee})$.
\end{prop}

\begin{proof} Denote after O'Grady 
$$
\tilde{\Delta}(0):=\{(A,v)\in
LG(10,\textstyle\bigwedge^3 W): \dim(F_v\cap \PP(A))= 2 \}.
$$ 
It was observed by O'Grady that
$\tilde{\Delta}(0)$ is smooth and is an open subset of $\tilde{\Delta}$ where we
have $\dim(F_v\cap \PP(A))\geq 2$. We know from \cite[Prop.~2.3]{O3} the description
of the tangent space to $\tilde{\Delta}$. In particular 
$T_{v_0}=\operatorname{Ker} \tau_{K}^{v_0}$, where $K:=\PP(A)\cap F_{v_0}$, in the notation of
\cite[eq.~(2.1.11)]{O3}. It remains to show that the linear space spanned by
$\pi_2(K)$ is perpendicular to  $\operatorname{Ker} \tau_{K}^{v_0}$.
To see this, note that $\pi_2|_{\PP(K)}$ is given by the system
of Pl\"{u}cker quadrics $\phi_v^{v_0}$ and use 
\cite[eq.~(2.1.11)]{O3}.
\end{proof}

\subsection{The case when  $\dim\Theta_A=0$}\label{Theta0}
The aim of this section is to study the case $\dim \Theta_A=0$ in Proposition
\ref{og-conj} by showing that an EPW sextic $S_A$ with $\dim \Theta_A=0$ cannot be the adjoint hypersurface to
the birational image of a polarized IHS fourfold $(X,H)$ with $b_2(X)=23$ and $H^{4}=12$ satisfying condition \textbf{O}.

The closure in $V$  of the exceptional set of the restriction of the morphism $\pi$: 
$$  
V\dashrightarrow
S_A-\Big(\bigcup_{U\in\Theta_A}\mathbb{P}(U)\Big)
$$ 
is a reduced Weil divisor $E_G$ that
maps to the surface $\supp N_2$.
We also have exceptional sets of $\pi$ over points from
$\bigcup_{U\in\Theta_A}\mathbb{P}(U)$.
Since  $O_2\cap \mathbb{P}(A)$ is irreducible, we deduce that
there are two kinds of irreducible components of
the exceptional set of $\pi$: either
\begin{itemize}
\item[$\clubsuit$]
 one-parameter families of planes
such that the image through $\pi$ is a curve $C_0$ which is a component of $C_{U,A}\subset \PP(U)$, or
\item[$\spadesuit$] 
$3$-dimensional linear spaces $E_i$ for $i=1,\dots,s$ mapping to
points in $\C_{U,A}\subset S_A$ for some $U\in G(3,W)$.
\end{itemize}
We believe that such exceptional sets cannot exist. However, we only prove that the first type of exceptional set cannot occur (this is enough to complete the proof of Proposition \ref{og-conj}). For this we need to better understand the duality between $S_A$
and $S_A'$.
It would be nice to find a simpler proof of the following:

\begin{lemm}\label{noEx1}
The morphism $\pi$ has no exceptional set as  in  $\clubsuit$.
\end{lemm}

\begin{proof}
Suppose that such an exceptional  set $G'\subset V$ exists.
Denote by 
$$
G\subset \PP(A)\cap O_2
$$ 
the image of $G'$ under $\alpha$ such that
each fiber $G\supset G_v=G\cap F_v$ is a plane and $G$ maps to a curve $C_0\in
\PP(U_0)$ (which is a component of $C_{U_0,A}$).

We claim that $G_v$ intersects $T_{U_0}$ along a line contained in the determinantal cubic $E_{U_0}\subset T_{U_0}$. Indeed,
from the proof of Proposition \ref{normalV} it follows that the generic fiber
$G_v$ cannot be contained in $T_{U_0}$.
Next from \cite[Prop.~3.2.6 (3)]{O4} we infer that 
$G_v$ intersects the tangent
space $T_{U_0} \cap F_v$ only at $U_0$ and is disjoint from $\Theta_A$; thus
$C_0$ has a node at $v$.
The claim follows since the nodes on $C_0$ are at isolated points. We
also deduce that $C_0$ is a triple component of $C_{U,A}$, so it is either a multiple conic or
a line.

We infer that  $G\cap T_{U_0}$ has dimension $\geq 2$. From the proof of Proposition \ref{normalV} we know that $\dim(T_{U_0}\cap
\PP(A))\leq 3$, so either $G\cap T_{U_0}$ is a plane, or
$\dim(T_{U_0}\cap \PP(A))=3$.

Let us show that the second case cannot happen. Suppose that $\dim(T_{U_0}\cap \PP(A))=3$. Since $G_v\cap T_{U_0}$ is a line contained in the cubic $E_{U_0}$, we infer that $G\cap
T_{U_0}$ is a cone over a cubic curve $\mathcal{A}$ (which is a section of
$E_{U_0}$). Denote by $N$ a generic hyperplane section of $G$. Note that $N$ is
smooth because it maps under $\pi_1$ to a smooth curve with linear spaces as fibers. It
follows that $N$ is the projection of a rational normal
scroll that has $\mathcal{A}$ as $\PP^2$ section.  
This is only possible when $\mathcal{A}$ is reducible, but then $N$ should be reducible, a contradiction.
We deduce that $G\cap T_{U_0}$ is a plane.

\begin{claim} The support of the curve $C_0$ cannot be a line.
\end{claim}
\begin{proof} Suppose the contrary and fix a $v\in C_0$.
 Since the morphism $\pi_2|_{G_v}$ is given by a linear subsystem of conics with
base point $G_v\cap G(3,W)$, it is birational and contracts the line $G_v\cap
T_{U_0}$ to a point, we deduce that $\pi_2(G_v)$ is a surface which is an
irreducible quadric cone $Q_v\subset \PP^5$ tangent to $\PP(U_0^{\vee})$ along a
line with vertex at the image of the contracted line (because the image of a
line passing through $U_0$ on $G_v$ is a line passing through the image of
$G_v\cap T_{U_0}$).
Consider the rational scroll $N$ and denote by $f$ 
a generic fiber of $\pi_1|_N$ and by $c_0$ the
section $T_{U_0}\cap N$. We saw that $c_0$ is a line (since $G\cap T_{U_0}$ is a plane).
We have $H|_{N}=f$ and $H_2|_{N}=a.f+b.c_0$ for some $a,b\in
\index{}\mathbb{Z}$.

We have two possibilities: $\pi_2(G)$ is either a quadric surface or a
threefold. Let us treat the first case.
Suppose that $Q_{v_1}$ and $Q_{v_2}$ are equal for $v_1\neq v_2$. Since $C_0$ is a line, $H|_{c_0}$ has degree~$1$.
Next, from $2T=H+H_2$ and  $\pi_2(c_0)\subset \PP(U_0^{\vee})\cap Q_{v_0}$ we infer that $ H_2|_{c_0}$ has degree $2\deg
c_0-1\leq 2$. So $c_0$ is a line. It follows that $N\subset \PP(A)$ is embedded
by $c_0+(e+1)f$ where $c_0^2=-e$ on $N$. Observe that $\pi_2|_N$ has connected linear fibers which are linear sections of the spaces $F_v'$. On the other hand, 
$\pi_2(N)=\pi_2(G)=Q_{v}$ so $2=(H_2|_{N})^2$ because $\pi_2|_N$ is birational. So using $2T=H+H_2$ we infer
$H_2=2c_0+(2e+1)f$, contradicting $4(2e+1)=(H_2|_{N})^2$.

It follows that the dimension of $\pi_2(G)$ is $3$, and $\pi_2|_N$ is birational.
One should have in mind that $\pi_2|_G$ is an isomorphism outside the singular
locus 
$$
\mathbb{G}=G'\cup \bigcup_{U\in \Theta_A}\PP(U^{\vee})
$$ 
of $S_A'$. From
Proposition \ref{dual-tangent} the tangent line $T_r{C_0}$ to $C_0$ at $r \in C_0$
is projectively dual to the space $\PP^3_r$ spanned by $\pi(G_r)=Q_r$. We have assumed
that $C_0$ is a line, so the image of $\pi_2(G)$ is a projective space that we
denote by $\PP$.
Since the double point locus of $ S_A'$ is of codimension~$2$, we infer that
$\pi_2|_{G}$ is birational.
Consider the locus $\mathbb{G}'$ of points $p\in \PP$ such that there are two
different $v_1,v_2\in C_0$ with $p \in Q_{v_1}\cap Q_{v_2}$ and
$\mathbb{G}'\subset \mathbb{G}$. We shall obtain a contradiction by proving that
$\mathbb{G}'=\PP$.
Fix a generic $v_0\in C_0 $; it is enough to prove that $Q_{v_0}\subset
\mathbb{G}'$.  
When $v\in C_0$ varies, the center of the cone
$Q_v$ moves along a curve in $\PP(U_0^{\vee})\subset \PP$ such that $Q_v$ is
tangent to $\PP(U_0^{\vee})$. We conclude by observing that  such quadrics cannot be in the same pencil determined by a common quartic curve.
\end{proof}

We deduce that $C_0$ is a triple conic and $T_{U_0}\cap \PP(A)$ is a plane. Consider again the ruled surface $N$ such that $c_0$ is a line and $N\subset
\PP(A)$ is embedded by $c_0+(e+1)f$ for some $e\in\mathbb{Z}$. Then
$H|_{N}=2f$ so $H_2|_N=2c_0+2e.f$. On the other hand, using again Proposition
\ref{dual-tangent} we see that $\pi_2(G)\subset \PP(W^{\vee})$ is contained in a quadric hypersurface
$\mathcal{Q}$ of rank $3$.
More precisely, $\mathcal{Q}$ is a cone, with a plane $\PP(U_0^{\vee})$ as
vertex, over a conic \textit{curve} $\mathcal{W}$ such that $\mathcal{Q}$
is covered by projective spaces $\PP^3_r$ dual to the tangent lines to $C_0$.
It follows that
$\pi_2|_G$ is an isomorphism outside $G\cap T_{U_0}$.
Consider the pull-back by $\pi_2|_N$ of a generic hyperplane containing
$\PP(U_0^{\vee})$. Since the intersection of the
hyperplane with $\mathcal{Q}$ are two projective spaces,  the class
of the pull-back $H_2|_N$ is $a.c_0+2.f$. Using $2T=H+H_2$ (see Lemma \ref{W_2})
we compute that $a=2$ and
$e=1$, thus $N$ is the blow-up of $\PP^2$ in one point with $c_0$ as exceptional line. Moreover, $\pi_2|_N$ contracts $c_0$ and maps $N$ to a projective plane. We infer that $\pi_2(N)$ intersects $\PP(U_0^{\vee})$
at only one point which is the image of $c_0$.
It also follows that $\pi_2(N)$ is either the second Veronese embedding of $\PP^2$ or a smooth
central projection of this second Veronese (because $\pi_2(N)$ can be singular only at one point).
Consider the curve $D_0$ which is the generic fiber of the projection of $\pi_2(N)$ with center $\PP(U_0^{\vee})$ to the curve $\mathcal{W}$. The curve $D_0$ can be seen as the intersection $\pi_2(N)\cap \PP_v^3$ for some
generic $v\in C_0$. Since there are no lines or degree three curve contained in the projection of the double Veronese, and a hyperplane section intersects $\pi_2(N)$ along a degree $4$ curve, we
deduce that $D_0$ is an irreducible plane conic.
We obtain a contradiction since a smooth conic $D_0\subset Q_v=\pi(G_v) \subset \PP_v^3$ cannot contain the
center of the cone $Q_v$.
\end{proof}

We can now return to the proof of Proposition \ref{og-conj}. We showed that the exceptional locus of $\pi$
consists of $3$-dimensional linear spaces $E_i$ for $i=1,\dots,s$ mapping to
points in some $\C_{U,A}\subset S_A$ for some $U\in G(3,W)$.
To obtain a contradiction we proceed  as in the general case.
By \cite[\S 1.2.2]{D} the rational map between the sextic $S_A$ and its dual
$S_A'$ is given by the partial derivatives
of the sextic $s_A$ defining $S_A$.
The composition
\[
V\xrightarrow{\pi} S_A\dashrightarrow S_A'\subset \PP(W^{\vee})
\]
is given by the linear system induced by the pull-back of quintics which are the
partial derivatives of $s_A$ on $V$.
On the other hand, by Remark \ref{alpha}, each such generic quintic $q'$
corresponds to an irreducible Cartier divisor $Q'\in|2T-H|$
on $V$.
The divisor $Q'$ coincides with the proper transform of the zero locus $\{ q'=0\} \cap S_A$ on $V$ (they are equal on an open subset of~$Q'$).
Recall that $S_A$ has ordinary double points along a generic point of $\supp
N_2$. It follows from  Lemma \ref{noEx1} that
$$
\pi^{\ast}(Q')=E_G+\sum a_iE_i +B,
$$ 
where $a_i\geq 0$,
$B\in|2T-H|$ is an effective Cartier divisor on the
normal variety $V$, $E_i$ are exceptional divisors mapping to points on the singular locus $C\subset X'$, and
$E_G$ is the exceptional divisor over $\supp N_2$. We infer 
$E_G+\sum
a_iE_i$ is a Cartier divisor in the linear system $|6H-2T|$.

By Proposition \ref{ciC} we find,  as in the general case, a divisor
$D\subset V$ in the linear system $|3H-T|$ that maps to $C\subset S_A$.
From Proposition \ref{N3 5.1} we deduce that $D$ is decomposable such that $D-E_G$
is an effective Weil divisor.
 We infer that
$$
D-\Big(E_G+\sum a_iE_i\Big)
$$ 
is a Cartier
divisor in the linear system $|3T-3H|$, denote it by $D'$. Since the Weil divisor $E_i$ intersects
$\alpha^{-1}(U)$ in isolated points, we infer that $D'$ restricts to an effective
curve on the plane $\alpha^{-1}(U)$, where $U\in \Theta_A$ is fixed. On the
other hand,  $\oo_V( T)|_{\alpha^{-1}(U)}=\mathcal{O}_{\alpha^{-1}(U)}$ and
$$\oo_V(H)|_{\alpha^{-1}(U)}=\mathcal{O}_{\alpha^{-1}(U)}(1).$$ Thus the
restriction of a divisor from
$|3T-3H|$ cannot be an effective curve on
$\mathbb{P}(U)$ (see \cite[Prop.~1.35(1)]{KM}). It follows that $ D$ contains
$\alpha^{-1}(U)$,
so $X'$ contains $\mathbb{P}(U)$,  a contradiction by Proposition
\ref{propC_(U,A)}. It follows that the adjoint sextic $S_A$ has $\dim \Theta_A\geq 1$.

\subsection{The case when  $\dim \Theta_A\geq 2$}\label{Theta2}

In this section we
consider adjoint EPW sextics with $\dim \Theta_A\geq 2$.
We show that such a sextic has to be very special as described in  Proposition \ref{og-conj}.

\subsubsection{$\dim \Theta_A\geq 3$}
We show first that $\dim \Theta_A\geq 3$ cannot happen.
Choose an irreducible component
$\Theta'_A$ of $\Theta_A$. Denote by $$\G=(\pi(\alpha^{-1}(\Theta_A')))_{red}$$ the reduced sum of the planes $\mathbb{P}(U)$ for
$U\in \Theta'_A$.

\begin{lemm}\label{lemdim} If $\Theta_A'$ has dimension~$k$ and $\G$ has
dimension $\leq k+1$ then there is a point $U\in
\Theta'_A$ such that $\C_{U,A}$ is a plane.
\end{lemm}

\begin{proof}
 First, $(\alpha^{-1}(\Theta_A'))_{red}$ is irreducible of dimension $k+2$, so
the image $\G$ is irreducible. Suppose it
has dimension $\leq k+1$ and  all the $\C_{U,A}$ are curves (outside these curves the fibers of $\pi$ are points). Then there
exists an open set $\mathcal{U}\subset (\alpha^{-1}(\Theta_A'))_{red}$ such that $\pi|_{\mathcal{U}}$
is $1:1$ onto a proper subset of $\G$, a contradiction since
$(\alpha^{-1}(\Theta_A'))_{red}$ is irreducible.
 \end{proof}
 
Since $\dim \G\leq \dim S_A \leq 4$ and $\dim \Theta_A\geq 3$, we infer that $X'\subset\PP(W)$ has to contain a
plane, contrary to  condition \textbf{O}.

\subsubsection{$\dim \Theta_A=2$}
 The strategy in this case is to show that in many cases the support $\C_{U,A}\subset \PP(U)$
has degree $\geq 4$. Then we apply several times the following:

\begin{lemm}\label{cubic} If  $\PP(U)\cap X'\subset \PP(W)$ has dimension $1$
then it supports a cubic curve.
\end{lemm}
\begin{proof} If $\dim \Theta_A\leq 1$ then the assertion is a consequence of Proposition
\ref{ciC}. If $\dim \Theta_A=2$ similar arguments apply:
For a fixed $U\in \Theta_A'$ the plane $\alpha^{-1}(U)\subset
\PP(\Omega_{\PP^5}^3(3))$ is a plane that maps under $\pi$ to $\PP(U)$.
On the other hand,  $\alpha^{-1}(U)$ is contained in $\PP(10\oo_{\PP^5})$
such that $\pi^{\ast}(\oo_{\PP^5}(1))$ is equal to the pull-back of
$\oo_{\PP^5}(1)$ on $\PP(10\oo_{\PP^5})$ and
$$\mathcal{O}_{(10\mathcal{O}_{\mathbb{P}^5})}(1)|_{\alpha^{-1}(U)}=\mathcal{O}_{
(\Omega^3_{\mathbb{P}^5}(3))}(-1)|_{\alpha^{-1}(U)}.$$ Thus we can conclude as in
Proposition \ref{ciC}.
\end{proof}

O'Grady observed also that we can apply the Morin theorem \cite{M}.
 Indeed, if $\Theta'_A$ is an irreducible component of $\Theta_A$ of dimension
$\geq 1$ then it parameterizes mutually intersecting planes in $\PP(W)$.  By the Morin
theorem,  
 $\Theta'_A$ is then a linear section of one of the following sets:
\begin{enumerate}
\item $\mathbb{P}^3$ embedded in $G(3,W)\subset \mathbb{P}(\bigwedge^3 W)$ by the double
Veronese embedding,
 \item $G(2,5)\subset F_v \subset G(3,W)$ embedded by the closures of fibers of $\pi_1$,
 \item $G(2,5)\subset F'_v\subset G(3,W)$ embedded by the closures of fibers of $\pi_2$,
\item $T_P\cap G(3,W)$ where $T_P$ is the projective tangent space at $P$ to
$G(3,W)\subset \mathbb{P}(\bigwedge^3 W)$,
\item $\mathbb{P}^2$ embedded in $G(3,W)\subset \mathbb{P}(\bigwedge^3 W)$ by
the triple Veronese embedding.
\end{enumerate}
 In order to complete the proof of Proposition \ref{og-conj} we check case by case the possible two-dimensional irreducible
components $\Theta'_A$ of $\Theta_A$ and find that either:
\begin{itemize}\item[(I)] the adjoint EPW sextic $S_A$ is a double determinantal cubic, or
\item[(II)] the EPW sextic $S_A\subset \PP(W)$ has a non-reduced component supported on a
hyperplane.
\end{itemize} 
In  case (I),  $\Theta'_A$ is the third Veronese embedding of $\mathbb{P}^2$ in $G(3,W)\subset \PP(\bigwedge^3W)$.
Case (II) happens for example when $\Theta_A'$ is a
plane.
Note that by Lemma \ref{lemdim} we can assume that $\G$ is a hypersurface of
degree $\leq 3$ (because $\G$ is a non-reduced component of $S_A$).
Let us study using  Lemma \ref{cubic} each case of the Morin theorem separately:

\medskip
\noindent
\textbf{Case (1)} From Lemma \ref{lemdim} we deduce that $\Theta_A'$ is a
hyperplane section of the double Veronese embedding of $\mathbb{P}^3$ (this is
the only possibility because there are no planes contained in this double
Veronese).
It follows from \cite[Claim 1.14]{O2} that $\G=(\pi(\alpha^{-1}(\Theta_A')))_{red}$ is a smooth quadric, and we
have the following:
 \begin{itemize}\item from
\cite[Prop.~2.1]{O5} it follows that $\G$ has multiplicity $2$ in the EPW sextic
$S_A$ (thus $S_A$ can be written in the  form $2\G+R$ where $R$ is a
quadric),
\item  $R\cap \G$ is contained in the sum of
$\C_{U,A}$ for $U\in \Theta_A'$ (because the sextic can be more singular only along such curves),
\item the restriction of $\pi\colon(\alpha^{-1}(\Theta_A'))_{red}\to \G$
is the blow-up of a plane $F$ contained in $\G\cap R$ ($\alpha^{-1}(\Theta_{A}')\to \Theta_A'$ is the restriction of 
$\PP(\Omega^1_{\PP^3}(2))\to \PP^3$ and $\pi|_{\alpha^{-1}(\Theta_{A}')} $ is given by the system $\oo_{\alpha^{-1}(\Theta_{A}')}(1)$).
 \end{itemize}  
Since the curves $\C_{U,A}$ cover $F$, we have $F\subset X'$.
Since each curve $\C_{U,A}$ is contained in $X'$, this  contradicts
condition \textbf{O}.

\medskip
\noindent
\textbf{Case (2)}  The
planes parameterized by $\Theta_A'$ contain the point $v$ and are defined by a
line $l_p\subset
G(2,V/[v])$.
Using \cite[Prop.~2.31]{O2} we deduce that $\Theta_A'$ is either
 \begin{itemize}
 \item[(a)] a plane or $\Theta_A'\subset G(2,T)\subset G(2,5)$ where $T\in G(4,5)$,
or
 \item[(b)] $\Theta_A'$ is a linear section of $G(2,5)$ which is a del Pezzo surface,
or
 \item[(c)] there is a line $l_0\subset \PP(V/[v])$ that intersects all the lines
$\PP(V/[v])$ parameterized by $\Theta_A'$.
 \end{itemize} 
We shall treat each case separately.

Assume (a); then the planes parameterized by $\Theta_A'$ cover a
hyperplane. This hyperplane has to be a
multiple component of $S_A$, so we are in case (II).

Assume (b), so that $\Theta_A'$ is a linear section of $G(2,5)\subset F_v$.
Then $\Theta_A'$ is a possibly singular del Pezzo surface $D_5$ of degree $5$
(observe that $D_5$ cannot be reduced if it has one component because of the
degree). Then the sum of the planes parameterized by $\Theta_A'$ is a cone over a
cubic hypersurface; denote it by $\mathbf{Q}$. More precisely, these planes are
spanned by the lines corresponding to points on $D_5\subset G(2,5)$ (the sum of
these lines is a cubic threefold, denote it by $\mathbf{Q}'\subset \PP(V/[v])$).
It follows that the corresponding EPW sextic is a double cubic.
Since $\dim (\PP(A)\cap F_v)=5$, it follows from \cite[Prop.~3.1.2]{O4} and
\cite[Claim~3.2.2]{O4} that $v$ is a point of multiplicity $6$ on
$C_{U,A}$ for $U\in D_5$.
Thus $C_{U,A}$ is a sum of multiple lines passing through $v$ (if it is the
whole plane we obtain a contradiction).

Let us now identify the sets $\mathcal{B}(U,A)$ in order to prove that $C_{U,A}$
has to be reduced for a generic $U\in D_5$.
Let us fix such a generic point $U$ of $D_5$; then 
$\mathbb{P}(A)\cap T_{U,G(3,W)}$ has dimension~$2$. Moreover, $\dim (F_v\cap
\mathbb{P}(A)\cap T_{U,G(3,W)})=2$ because this space contains the tangent space
to the del Pezzo surface $D_5\subset F_v$ and is contained in the previous
intersection. It also follows
that the set of $w\in \mathbb{P}(U)$ such that 
$$\dim(\mathbb{P}(A)\cap F_w \cap
T_{U,G(3,W)})\geq 1
$$ 
is the singleton $\{v\}$.
Since $D_5$ is irreducible of
dimension $2$, we infer that $U$ does not belong to any line on $D_5\subset \mathbb{P}^5$
 (such lines cannot cover the whole $D_5$).
Thus  for $U'\in D_5- \{ U \}$ we have $\mathbb{P}(U')\cap
\mathbb{P}(U)=\{v\}$.

So  $\mathcal{B}(U,A)$ is  the 
sum of the intersections
$\mathbb{P}(U)\cap \mathbb{P}(V_0)$ where $V_0\in \Theta_A-D_5$ and $\{ v\}$.

For a fixed $V_0$, $\mathbb{P}(V_0)$ intersects $\mathbb{P}(U)$
outside $v$
(because $F_v\cap G(3,W)=G(2,5)$) and from Lemma \ref{property-SingC_U} in
one point (since $\C_{U,A}$ is a sum of lines passing through $v$).
Since the plane $\mathbb{P}(V_0)$ has to be contained in our cubic hypersurface
$S$, the set $\C_{V_0,A}$ must be the whole $\mathbb{P}(V_0)$.

It follows that $C_{U,A}$ is a reduced sum of six lines for a generic choice of
$U\in D_5$. We deduce that for each such $V_0$ we have $\C_{V_0,A}=\PP(V_0)$, contradicting condition \textbf{O}.

Assume (c); then $\Theta_A'$ is a linear section of the cone with vertex $U_0$
over the Segre embedding $\mathbb{P}^1\times\mathbb{P}^2$.
The planes parameterized by points in $\Theta_A'$ are spanned by the point $v$
and a line in $\mathbb{P}(V/[v])$. 
More precisely, the line in $\PP(V/[v])$ is described as follows: the first factor of $\mathbb{P}^1\times\mathbb{P}^2$ corresponds to a
choice of a point on the line $l_0$, and the second factor corresponds to a
choice of a plane containing $l_0\subset \mathbb{P}^4$; finally the directrix of our
cone with vertex $U_0$ gives a choice of a line on this plane passing through our point.

We will obtain a contradiction by showing that $\mathbb{P}(U_0)$ must be
contained in $X'$. Thus it is enough to show that the sum of the curves $\C_{U,A}$
for $U\in \Theta_A'$ covers the line $l_0$.
By Lemma \ref{property-SingC_U} it is enough to prove that for each point of
$l_0$ there are at least two lines parameterized by $\Theta_A'$ that contain
this point.

 If $\Theta_A'$ contains $U_0$ then it is a cone and we obtain a
contradiction unless $\Theta_A'$ is a plane spanned by $U_0$ and a line
contained in the second factor of $\mathbb{P}^1\times\mathbb{P}^2$. Indeed, the
planes in $\mathbb{P}(W)$ parameterized by the point from $\Theta_A'$ intersect
in this case along a line spanned by $v$ and the fixed point from $l_0$ and
cover a hyperplane.

If $\Theta_A'$ does not contain the vertex $U_0$, we obtain a
contradiction similarly unless the image of the projection 
$$
\PP^1\times \PP^2\supset \Theta_A'\to \mathbb{P}^1
$$ 
is
a point.
Suppose the the image of the projection above is a point that we denote by $Q_0$. Then $\Theta_A'$ is a plane. Next the planes parameterized by
$\Theta_A'$ pass through a line $l$ (determined by $v$ and $Q_0$) and cover a
hyperplane $H_0$ which is a non-reduced component of $S_A$, so we are in case (II).

\medskip
\noindent
\textbf{Case (3)} Suppose that $G(2,5)$ is equal to  $F_v\cap G(3,W)$ for some $v\in W$.
This embedding is given by choosing a point $L\in G(5,W)$ that gives a natural
embedding $G(3,L)\subset G(3,W)$.
In this case the sum of the planes corresponding to points in $\Theta_A'$ is
contained in the hyperplane $\mathbb{P}(L) \subset \mathbb{P}(W)$.
By Lemma \ref{lemdim} we can assume that this sum covers $\mathbb{P}(L)$. It follows from 
\cite[Cor.~1.5]{O2}
that $S_A$ has a non-reduced linear component; so we are in case (II).

\medskip
\noindent
\textbf{Case (4)} Then from Lemma \ref{lemP2xP2} the component $\Theta_A'$ is a 
two-dimensional linear section of the cone over $\mathbb{P}^2 \times
\mathbb{P}^2$ in $\mathbb{P}^9$ with vertex $U_0$. It is useful to have in mind
the description of the family of planes parameterized by $\Theta_A'\subset \PP^2\times \PP^2$:

\begin{lemm}
 Geometrically the first factor of $\mathbb{P}^2 \times \mathbb{P}^2$
corresponds to a choice of a line in $\mathbb{P}(U_0)$ and the second factor to the
choice of
a $\mathbb{P}^3$ containing $\mathbb{P}(U_0)$. The directrix of the cone
corresponds to planes containing the fixed line in a fixed $\mathbb{P}^3$.
\end{lemm}

Suppose first that $\Theta_A'$ contains the vertex of the cone $U_0\in G(3,W)$.
Then the plane $\mathbb{P}(U_0)$ is covered by the intersection with other
planes corresponding to points from $\Theta_A'$ unless $\Theta_A'$ maps to a
point under the projection $\PP^2\times \PP^2 \supset \Theta_A'\to \mathbb{P}^2$. Thus, in the first case, we
obtain a contradiction from Proposition \ref{propC_(U,A)}. But in the second
case we see that $\Theta_A'$ is a plane; then we are in Case (2) that was
described before.

We can assume that $\Theta_A'$ does not contain the vertex of the cone so we can
use \cite[Prop.~2.33]{O2}. We want to obtain a contradiction by showing that
$\mathbb{P}(U_0)\subset X'$. For this it is enough to see that the sum of the
curves $C_{U,A}$ for $U\in \Theta_A'$ contains $\mathbb{P}(U_0)$.
Consider the projections to the factors $\mathbb{P}^2\leftarrow \Theta_A'\to
\mathbb{P}^2$ (recall that $\Theta'_A\subset \PP^2\times \PP^2$ ). Since by Lemma \ref{property-SingC_U} the intersection of two
planes $\mathbb{P}(U)$ and $\mathbb{P}(V)$ is contained in the curve $C_{U,A}$, we obtain a contradiction when the dimensions of the images of
both projections have dimension $\geq 1$. The remaining case is when
$\Theta_A'={v}\times \mathbb{P}^2$, where $v$ corresponds to a fixed line in
$\mathbb{P}(U_0)$. But then we are in Case (2).

\medskip
\noindent
\textbf{Case (5)} We assume that $\Theta'_A$ is the triple Veronese
embedding of $\mathbb{P}^2$. Then from \cite[Claim 1.16]{O2} we know that  $\G=(\pi(\alpha^{-1}(\Theta_A')))_{red}$
is the secant cubic of the Veronese
surface in $\mathbb{P}^5$. It follows from \cite[\S 4.4]{O4}
that for all $U\in \Theta'_A$ the set $C_{U,A}$ is a triple smooth
conic. Consider the restriction $\mathcal{E}_{\Theta}\to \Theta'_A$
of the tautological bundle on $G(3,W)$. In this case we obtain
$\mathcal{E}_{\Theta}=S^ 2\Omega^1_{\mathbb{P}^2}(1)$ and the
following diagram:
\[\mathbb{P}(\Omega^3_{\mathbb{P}^5}(3))\supset \mathbb{P}(S^
2\Omega^1_{\mathbb{P}^2}(1))\xrightarrow{f} \Theta'_A\subset
\mathbb{P}(\textstyle\bigwedge^3W)\]
$$\ \ \downarrow \pi \ \ \ \ \ \ \ \ \ \ \ $$
$$\mathbb{P}^5\supset \G \ \ \ \ \ \ \  \ \ \ \ \ \ $$
The system of quadrics containing the Veronese surface gives the Cremona
transformation  
\begin{equation}\label{diag2}
\begin{array}{ccccc}
  \mathbb{P}^5 &&\dashleftarrow&& \mathbb{P}^5\\
&\nwarrow^{c_1}&&\ \ \ \nearrow_{c_2}&\\
&& \mathbb{K}&&
\end{array}
\end{equation} 
where $c_1$ and $c_2$ are the blow-ups of the Veronese surface
$V_i\subset\mathbb{P}^5$ for $i=1,2$ respectively. Then the exceptional
 divisor $\mathbb{E}$ of $c_1$ maps under $c_2$ to the determinantal cubic
singular along $V_2$. Moreover, the exceptional divisor $F$ of the induced map
$\mathbb{E}\to \G$ is naturally isomorphic
to the projective bundle $\mathbb{P}(\Omega^1_{V_2}(1))$.
We also see that $\pi|_{\mathbb{P}(S^ 2\Omega^1_{\mathbb{P}^2}(1))}$ can be seen
as the blow-up of $\G$ along its singular locus, thus we can identify it with
$c_2|_{\mathbb{E}}$.

We deduce from the diagram (\ref{diag2}) that  $(2H-F)=2B$ on
$\mathbb{P}(S^2\Omega^1_{\mathbb{P}^2}(1))$ where $B$ (resp.~$H$) is the pull-back of the hyperplane from $\mathbb{P}^2=\Theta_A'$ (resp.~$\mathbb{P}^5$).
 The linear system $|3H+T|$ can be seen on $\mathbb{E}$ as $|3H+3B|$.
By Proposition \ref{N3 5.1} we infer that $3H+3B- F$ is effective, so it is an
element of $|H+5B|$.

We can go in the other direction: choose an element from $|H+5B|$, map it to $\G$
and choose a hypersurface of degree $12$ singular along the image.
 Since the conductor locus is non-reduced, the singularities of this hypersurface have
to have generically tacnodes (see \cite[\S 4.4]{Re}) along the intersection with
$S_A$. This can lead to a possible counterexample to the O'Grady conjecture.

\begin{rem}
 Let us describe more precisely the EPW sextic $S_A$ in the missing cases when
$\Theta_A'$ is a plane.
First observe that if $\Theta_A'$ is a plane then it is contained in the tangent
space to $G(2,5)\subset F_v$ at one of its points; we can thus assume that we
are in  case (c) above.
  In this case $S_A$ is singular along a hyperplane $H_0$ which is a multiple
component such that there is a line $l\subset H_0$ contained in all the planes
$\PP(U)$ for $U\in \Theta_A'$. By Lemma \ref{property-SingC_U} the line
$l\subset H_0$ is also contained in all the curves $C_{U,A}$ for $U\in
\Theta_A'$. Moreover, the divisor $D\in |3H+T|$ from Proposition \ref{ciC} intersects
$\G=\alpha^{-1}(\Theta_A')_{red}$ (this is just the blow-up of $H_0$ along $l$)
along a divisor in the system $|4H-2E|+E$. So there is a
quartic on $H_0$ singular along $l$ that defines set-theoretically the
intersection of $H_0$ with the scheme $C$ defined by the conductor. So we can
describe the situation
(in the generic case) as follows: the EPW sextic is decomposable $2H_0 \cup Q$ such that
$Q$ is a quartic intersecting the hyperlpane $H_0$ along a quartic. The above quartic is singular along $l$.
Moreover, the intersection $H_0\cap Q$ supports the singular locus of $C\subset X' =\varphi (X)\subset \PP^5$.
Since $C$ has multiplicity $3$ at a generic point of the image,
 the hypersurface $X'\subset \PP^5$
has multiplicity $3$ along  $C$ and the singularities along $C$ are worse than
ordinary triple points (see \cite[\S 4.4]{Re}).
\end{rem}

\subsection{The case when  $\dim \Theta_A=1$}\label{Theta1} 

The aim of this section is to show
that the adjoint EPW sextic from Theorem \ref{EPW} cannot correspond to a
generic $A$ with $\Theta_A$ of dimension~$1$, i.e.~such that $\Theta_A$ is a line
(with some more conditions).
Following \cite[\S 2]{O2} we set
$$
\G = \Big(\bigcup_{P\in\Theta_A}\mathbb{P}(P)\Big)_{red},
$$ 
and we denote by $\mathcal{E}_{\Theta_A}\to \Theta_A\subset G(3,W)$ the
restriction of the tautological bundle from $G(3,W)$ and by
$f_{\Theta_A}\colon\mathbb{P}(\mathcal{E}_{\Theta_A})\to R_{\Theta_A}$ the
tautological surjective map. Observe that there is a natural
embedding of $\mathbb{P}(\mathcal{E}_{\Theta_A})$ in
$\mathbb{P}(\Omega_{\mathbb{P}^5}^3(3))$ (in fact into the exceptional
set $E\subset \PP(\Omega_{\PP^5}^3(3))$ described in Remark \ref {alpha}). The divisor $D\in |3H+T|$ (that maps to the conductor locus  $C \subset \PP(W)$) intersects
$\mathbb{P}(\mathcal{E}_{\Theta_A})$ along an effective divisor $D'$
that we shall analyze.

Suppose that $\Theta_A'$ is an irreducible component of $\Theta_A$.
O'Grady applied the Morin theorem to show that $1\leq\deg(\Theta_A')\leq 9$.
He also presented in \cite[Table 2]{O2} the precise description the corresponding curves
and of the corresponding three-dimensional 
sets $\G$.

If $\deg \Theta_A=1$ then $\Theta_A$ is a line that we denote by $t$. Then the variety
$\G$ is a $3$-dimensional linear space containing a line $l$
such that the exceptional divisor $E'$ of $f_{\Theta}$ (in fact
$f_{\Theta}$ is the blow-up along $l$) maps to $l$. We compute that
on $\mathbb{P}(\mathcal{E}_{\Theta})$ we have $T=H-E'$ so $D'=4H-E'$.
Since the planes $\mathbb{P}(P)\subset \PP(W)$ contain $l$ and $\C_{P,A}\subset \PP(P)$ 
cannot be a plane, we deduce that the image of $D'$ on $\G$ is
an irreducible quartic containing $l$ or a sum of two quadrics (if there is a
plane component we obtain a contradiction with \textbf{O} because this component has to be contained in
$X'\subset \PP(W)$).

On the other hand, let us analyze the reduced sum $\mathcal{Z}\subset
\G$ of the curves $C_{P,A}\subset\mathbb{P}(P)$ for $P\in
\Theta_A$. As observed before, we have $\mathcal{Z}\subset \supp D'$. Observe
that generically $C_{P,A}$ is a sum of a reduced quartic and a double
line $l$, so we obtain a contradiction in this case. The problem is
the special choices of $A$. There are a lot of possibilities; we hope to consider
them in a future work.

\section{Appendix}\label{App}

Let $W$ be a $6$-dimensional vector space. The exterior product defines a
symplectic form on the 20-dimensional vector space $\bigwedge^3 W$. The natural action of $PGL(W)$ on
$\mathbb{P}(\bigwedge^3 W)$ has four orbits
$\mathbb{P}(\bigwedge^3 W) \setminus O_1$, $O_1\setminus O_2$, $O_2
\setminus O_3$ and $O_3$, where $O_1\supset O_2\supset O_3$ are subvarieties of
dimensions $18$, $14$, and $9$.
Moreover, it is known that $O_3=G(3,W)$, $O_1$ is a quartic described in
\cite[Lem.~3.6]{Do} and $O_2$ (resp.~$O_3$) is the singular locus of $O_1$
(resp.~$O_2$). In this paper we are only interested in the orbits $O_3\subset O_2$.

The locus $O_2\subset \mathbb{P}(\bigwedge^3 W)$ can be seen as the set of points
lying on more than one chord of $G(3,W)\subset \PP(\bigwedge^3 W)$ (see \cite[Lem.~3.3]{Do}) or as the union
of 
all
spaces spanned by some $G(3,N)$ for $N\subset W$ of dimension 5, which is
equal to the union of 
all
 spaces spanned by some flag variety $F(p,3,N)$ for some $p\in W$. With
this interpretation we get a description of $O_2$ as the set of $3$-forms
$$
\{[\alpha \wedge \omega] \in \mathbb{P}(\textstyle\bigwedge^3 W) \ | \ \alpha \in W,\
\omega \in
\textstyle\bigwedge^2 W\}.
$$
It follows that there are two natural fibrations of $\pi_1,\pi_2 \colon
O_2\setminus O_3\rightarrow \mathbb{P}^5$ such that the closures of the fibers
are $9$-dimensional linear spaces. More precisely, $\pi_1$ is defined as the map
$$
O_2\setminus O_3\ni [ \alpha \wedge \omega] \mapsto [ \alpha] \in \mathbb{P}(W)
$$ 
and
$\pi_2$ is the map 
$$
O_2\setminus O_3\ni [\alpha \wedge \omega] \mapsto
[\alpha\wedge\omega\wedge\omega]  \in \mathbb{P}(W^{\vee}).
$$

\begin{lemm}\label{well-defined} 
The maps $\pi_1$ and $\pi_2$ are well defined on $O_2 \setminus
O_3$.
\end{lemm}

\begin{proof}
Assume that $[\alpha_1\wedge \omega_1]=[\alpha_2\wedge \omega_2]\in O_2\setminus
O_3$
for some  $\alpha_1, \alpha_2 \in V$ and $\omega_1, \omega_2 \in \bigwedge^2 W$.
We need to show that $[\alpha_1]=[\alpha_2]$ and 
$$
[\alpha_1\wedge \omega_1\wedge
\omega_1]=[\alpha_2\wedge \omega_2\wedge \omega_2].
$$
Observe that under our assumption we have $\alpha_1\wedge \alpha_2\wedge
\omega_2=0$, but $\alpha_2 \wedge \omega_2$ is not a simple form,
hence $\alpha_1 \wedge \alpha_2=0$ and the first part of the assertion follows.
We infer the second part since
\[
[\alpha_2 \wedge \omega_2 \wedge \omega_2]=[\alpha_1\wedge \omega_1\wedge
\omega_2]=[\alpha_1\wedge\omega_2\wedge\omega_1]=
[\alpha_2 \wedge \omega_2 \wedge \omega_1]=[\alpha_1 \wedge \omega_1 \wedge
\omega_1].\qedhere
\]
\end{proof}

\begin{prop}\label{Pic(W_2)}
The divisor class group of $O_2$ has rank $2$ and is
generated by the closures of the pull-backs of the hyperplane sections by
$\pi_1$ and $\pi_2$;
denote them by $H$ and $H_2$.
\end{prop}

\begin{proof}
 First, the Picard
group of the projectivized vector bundle
$$
\mathbb{P}(\Omega^3_{\mathbb{P}^5}(3))\subset
\mathbb{P}(\textstyle\bigwedge^3W)\times \mathbb{P}^5
$$ 
has rank $2$ and is
generated by $H$ and $T$, the pull-backs of hyperplanes from
$\mathbb{P}(W)$ and $\mathbb{P}(\bigwedge^3W)$ respectively.
So it is enough to consider the map 
$$
\alpha \colon
\mathbb{P}(\Omega^3_{\mathbb{P}^5}(3))\to O_2\subset \PP(\textstyle\bigwedge^3 W)
$$ 
given by the linear system of the big divisor
$T$.
By \cite[Thm.~1]{RS}, the divisor class group of $O_2\subset \PP^{19}$ is
isomorphic to the divisor class group of its generic codimension $10$ linear
section $O_2'$.
Since $O_2'$ is smooth, the latter is equal to the Picard group of $O_2'$.
On the other hand,
$\alpha$ restricted to the pre-image $O_2''$ of $O_2'$ is an isomorphism. Since
$O_2''$ is the intersection of ten generic big divisors from the system $|H|$,
we deduce from the generalized Lefschetz
 theorem \cite[Thm.~6]{RS} that the Picard group of $O_2''$ is isomorphic to the
Picard group of $\PP(\Omega^3_{\mathbb{P}^5}(3))$.
\end{proof}

Let us describe the projective tangent space to $O_2$ at a point $p\in O_2\setminus O_3$.
Denote first by $F_p=\overline{\pi^{-1}_1(\pi_1(p))}$ and $F'_p=\overline{\pi^{-1}_2(\pi_2(p))}$ the fibers of $\pi_i$ for $i=1,2$.

\begin{lemm}\label{styczna do W2}
Let $p=[\alpha \wedge \omega] \in O_2\setminus O_3$, where $\alpha \in W$ and
$\omega\in \bigwedge^2 W$. Then the projective tangent space $T_p O_2$ is
the linear space spanned by the two fibers $F_p$
and $F_p'$, passing through
$p$, and by the linear space 
$$
\Pi=\{[\gamma \wedge \omega] \in \mathbb{P}(\textstyle\bigwedge^3
W) \mid \gamma \in W \}.
$$
\end{lemm}

\begin{proof} It is clear that all three linear spaces are contained in $O_2$
and pass
through $p$. It follows that they span a subspace of the tangent space $T_p O_2$.
Recall that $O_2$
is of dimension $14$, and the intersection  $F_p\cap
F_p'$ is a $\mathbb{P}^5$. It follows
that the two fibers span a hyperplane in $T_p O_2$. It is hence enough to prove
that  $\Pi$ is not contained in the span of
the two fibers. To do so,  denote by $\Sigma_p$ the hyperplane 
$$
\PP (\{\beta
\in \textstyle\bigwedge^3 W \mid \beta \wedge \alpha \wedge \omega=0 \}).
$$
Clearly,  $F_p\cap F'_p \subset \Sigma_p$, whereas $\Pi\nsubseteq \Sigma$
as there exists $ \gamma \in W$ such that
$\gamma\wedge \alpha \wedge \omega \wedge \omega \neq 0$.
\end{proof}

\begin{rem}\label{remark on T p cap Sigma p}
Observe that $\Sigma_p\cap T_p O_2$ is the $\mathbb{P}^{13}$ spanned by the two
fibers.
\end{rem}

\begin{prop}\label{tgW_2}
 Let $T_p$ be the projective tangent space to $O_2$ at a smooth point $p\in O_2$. Then there
are no $5$-dimensional isotropic subspaces $K\subset T_p O_2$ such that $p\in K$
and 
$$
K\cap
F_p\cap F_p'\cap O_3=\emptyset.
$$
\end{prop}

\begin{proof} Let $K$ be an isotropic subspace of $T_p O_2$ and let $L$ be a
Lagrangian (maximal isotropic) subspace of $T_p O_2$ containing $K$.
Then, since $p\in K\subset L$, we have $L\subset \Sigma_p$, where $\Sigma_p$ is
as in  the proof of Lemma \ref{styczna do W2}. By Remark \ref{remark on
T p cap Sigma p}, we get
$K\subset L\subset  \PP(U_1) +
\PP(U_2)$. We observe that the projectivized support $S$
of the intersection form on the latter $\mathbb{P}^{13}$ has dimension $7$ and
is disjoint from $F_p\cap F'_p $.
It follows that $\dim(L\cap S)=3$, $\dim(L)= 9$ and
$F_p\cap F'_p \subset L$.
It is easy to see that
$F_p\cap F'_p \cap O_3$
is a quadric hypersurface in $F_p\cap F'_p $.
It follows that any $5$-dimensional subspace of $L$ meets
$F_p\cap F'_p \cap O_3$,
as it meets
$F_p\cap F'_p $ in a
line.
\end{proof}


\begin{lemm}\label{H_1+H_2} Let us keep the notation above. Then the linear system
$|H+H_2|$ is given by the
restrictions of quadrics to
$O_2\subset \PP(\bigwedge^3 W)$.
\end{lemm}

\begin{proof}
Let $v\in W^{\vee}$ and $\gamma\in W=(W^{\vee})^{\vee}$ correspond to the hyperplanes
$L_1\subset \PP(W)$ and $L_2\subset \PP(W^{\vee})$
respectively. Consider the quadric form
$$
Q \colon \textstyle\bigwedge^3 W \colon \omega \mapsto \omega(v)\wedge \omega \wedge \gamma
\in \bigwedge^{6}W=\mathbb{C}.
$$ 
It is enough to prove that
$Q^{-1}(0)\cap O_2 =\pi_1^{-1}(L_1)\cup \pi_2^{-1}(L_2)$, and this has
to be checked only outside $G(3,W)\subset O_2$.
\begin{itemize}
 \item We first prove the inclusion $\supseteq$. Take $\omega \in
\pi_1^{-1}(L_1)$. Then there exists
$\alpha \in  H$ such that $\alpha\wedge \omega =0$. We then observe that since
$\alpha \in H$, it follows that $\alpha\wedge \omega(v) =0$.
The inclusion of the second component follows by duality.
 \item Let us pass to the inclusion $\subseteq$. Take 
$$
\omega \in
O_2\setminus(\pi_1^{-1}(L_1)\cup \pi_2^{-1}(L_2) \cup G(3,W)).
$$ 
Then $\omega$
may be written in the form $\alpha \wedge \beta$ with $\beta\in \bigwedge^2W$
such that $\alpha\wedge \beta^2 \wedge w$ is non-zero and
$v(\alpha)$ is non-zero. The value of the quadric on $\omega$ is then the product
of these non-zero values.
\end{itemize}
\end{proof}

Denote by $G$ (resp.~$G'$) the singular locus of the EPW sextic $S_A\subset \PP(W)$ (resp. $S_A'\subset \PP(W^{\vee})$). It is known (see
\cite{EPW}) that $S_A$ has $A_1$ singularities along $G$ and that
$G\subset\mathbb{P}^5$ is a smooth surface of degree $40$.
It follows that the $G$ is scheme-theoretically defined by the six quintics
which are the partial derivatives of the sextic $S_A$.
Denote by $E,E_2\subset V':=O_2\cap \PP(A)$  (where $A\subset \bigwedge^3 W$ is a 10-dimensional Lagrangian subspace) the
exceptional locus of $\pi_i$ for
$i=1,2$ and by abusing notation $H$ the restrictions of $H$ to $V'\subset O_2$.

\begin{cor}\label{dol}
The morphism $\pi_1 \colon V'\rightarrow S_A$ is the blow-up of
$G\subset S_A$. Moreover, the birational map $
  \pi_2\colon V'\rightarrow S_A'$ is given by the linear system
$|5H-E|$.
\end{cor}

We also  obtain the following corollary (note that it can also be proved using the methods from \cite{W}):

\begin{cor}
 The degree of $O_2\subset \mathbb{P}^{19}$ is $42$.
\end{cor}

\begin{proof}

We have to compute $\frac{(6H-E)^4}{16}$. Thus it is enough to prove that
$H^4=6$, $H^3E=0$, $H^2E^2=-80$, $HE^3=-480$, and $E^4=-1344$.
First from the adjunction formula $E^2H^2=K_EH^2$, $E^3H=K_E^2H$, and
$E^4=K_E^3$.
Now from \cite[\S 4]{O} we deduce that $p\colon
E=\mathbb{P}(T_{G})\rightarrow G$. Thus $K_{E}=-2\psi$ where $\psi$ is
the tautological divisor. Finally, we need the  equality
$$
\psi^2-3\psi\cdot H+c_2(p^{\ast}(T_{G}))=0.
$$
\end{proof}

Since $E=2(3H-T)$ (see Lemma \ref{W_2}) is even in the Picard group of $V'$, there exists a double
cover of  $\overline{X}\rightarrow V'$ ramified along $E$ (we can take the
double cover ramified along $E_2$).
The strict transform of $E$ on $\overline{X}$ can be blown down so that the
image is the
irreducible symplectic manifold $X$ constructed by O'Grady.

\bigskip G. Kapustka Institute of Mathematics of the Polish Academy of Sciences,\\
ul. \'{S}niadeckich 8, P.O. Box 21, 00-656 Warszawa, Poland.\\
\\ Jagiellonian
University, {\L}ojasiewicza 6, 30-348 Krak\'{o}w, Poland.\\
\emph{E-mail address:} grzegorz.kapustka@uj.edu.pl\\

\bigskip M. Kapustka
 Department of Mathematics and Informatics,\\ Jagiellonian
University, {\L}ojasiewicza 6, 30-348 Krak\'{o}w, Poland.\\
\emph{E-mail address:} michal.kapustka@uj.edu.pl

\end{document}